\documentclass[table,a4paper]{article}

\usepackage[a4paper]{geometry}
\usepackage[english]{babel}
\usepackage[utf8]{inputenc}
\usepackage{amsmath}
\usepackage{amssymb}
\usepackage{amsfonts}
\usepackage{mathrsfs}
\usepackage{amsmath}
\usepackage{amsthm}
\usepackage{algorithm}
\usepackage{algpseudocode}
\usepackage{graphicx}
\usepackage{pgfplots}
\usepackage{pgfplotstable}
\usepackage{booktabs}
\usepackage{tabularx}
\usepackage{xcolor}
\usepackage{enumerate}
\usepackage{url}
\usepackage{mathdots}
\usepackage{mathtools}
\usepackage{cite}
\DeclareMathOperator{\lc}{lc}

\DeclareMathOperator{\Res}{Res}

\providecommand*{\donothing}[1]{}
\DeclareGraphicsExtensions{.pdf,.eps,.png,.jpg}
\usepackage{epstopdf}
\DeclarePairedDelimiter{\norm}{\lVert}{\rVert}

\DeclareMathOperator{\diag}{diag}

\definecolor{Gray}{gray}{0.95}

\pgfplotstableset{
  every head row/.style={before row=\toprule,after row=\midrule},
  clear infinite
}

\newtheorem{theorem}{Theorem}[section]
\newtheorem{proposition}[theorem]{Proposition}
\newtheorem{lemma}[theorem]{Lemma}
\newtheorem{corollary}[theorem]{Corollary}

\theoremstyle{definition}
\newtheorem{definition}[theorem]{Definition}

\theoremstyle{remark}
\newtheorem{example}[theorem]{Example}
\newtheorem{remark}[theorem]{Remark}

\title{Decay bounds for the numerical quasiseparable preservation in matrix functions\footnote{%
    This work has been partially supported by an INdAM/GNCS Research Project 2016,
    and by the Research Council KU Leuven, project
    CREA/13/012, and by the Interuniversity Attraction Poles
    Programme, initiated by the Belgian State, Science Policy Office,
    Belgian Network DYSCO.
}}
\author{ Stefano Massei\footnote{stefano.massei@sns.it},  \ Leonardo Robol\footnote{leonardo.robol@cs.kuleuven.be}\\ 
\\ $^\dagger$Scuola Normale Superiore, Pisa,\\
$^\ddagger$Department of Computer Science, KU Leuven}
\date{}

\renewcommand{\leq}{\leqslant}
\renewcommand{\geq}{\geqslant}

\DeclareMathOperator{\res}{Res}
	
\begin{document}
  \maketitle

  \begin{abstract}
    Given matrices $A$ and $B$ such that $B=f(A)$, where $f(z)$ is a
    holomorphic function, we analyze the relation between the singular
    values of the off-diagonal submatrices of $A$ and $B$. We provide
    a family of bounds which depend on the interplay between the
    spectrum of the argument $A$ and the singularities of the
    function. In particular, these bounds guarantee the numerical
    preservation of quasiseparable structures under mild
    hypotheses. We extend the Dunford-Cauchy integral formula to the
    case in which some poles are contained inside the contour of
    integration. We use this tool together with the technology of
    hierarchical matrices ($\mathcal H$-matrices) for the effective
    computation of matrix functions with quasiseparable arguments.
      
      \bigskip
      
       
       \bigskip 
       
  \end{abstract}
  
   
     \section{Introduction}
      Matrix functions are an evergreen topic in matrix algebra due to their wide use in applications \cite{gant59,lanc85,horn,high, gol12}. It is not hard to imagine why the interaction of structures with matrix functions is an intriguing subject. In fact, in many cases structured matrices arise and can be exploited for speeding up algorithms, reducing storage costs or allowing to execute otherwise not feasible computations. The property we are interested in is the {\em
        quasi-separability}. That is, we want to understand whether the submatrices of $f(A)$ contained in the strict
      upper triangular part or in the strict lower triangular part, called {\em
        off-diagonal submatrices}, have a ``small'' numerical rank.
      
      Studies concerning the numerical preservation of data-sparsity patterns were carried out recently  \cite{benzi2013decay,benzi2014decay,benzi-simoncini,canuto}. Regarding the quasiseparable structure \cite{vanbarel:book1,vanbarel:book2,eide2005,eide2013}, in \cite{gavrilyuk2002mathcal,gavrilyuk2005data,hackbusch2016hierarchical} Gavrilyuk, Hackbusch and Khoromskij addressed the issue of approximating some matrix functions using the hierarchical format \cite{borm}. In these works the authors prove that, given a low rank quasiseparable matrix $A$ and a holomorphic function $f(z)$, computing $f(A)$ via a quadrature formula applied to the contour integral definition, yields an approximation of the result 
      with a low quasiseparable rank. Representing $A$ with a $\mathcal H$-matrix and exploiting the structure
      in the arithmetic operations 
      provides an algorithm with almost linear complexity. 
      The feasibility of this approach is equivalent to 
      the existence of a 
      rational function $r(z) = \frac{p(z)}{q(z)}$ which well-approximates 
      the holomorphic function $f(z)$ on the spectrum of the
      argument $A$. More precisely, since the quasiseparable rank is
      invariant under inversion and sub-additive with respect to matrix addition and multiplication, if $r(z)$ is a good approximation
      of $f(z)$ of low degree then the matrix $r(A)$ is an accurate approximation of $f(A)$ with low quasiseparable rank. This argument explains the preservation of the quasiseparable structure, but still needs a deeper analysis which involves the specific properties of the function $f(z)$ in order to provide effective bounds to the quasiseparable rank of the matrix $f(A)$. 
      
      In this article we deal with the analysis of the quasiseparable structure of matrix functions by studying the interplay
      between the off-diagonal singular values of the matrices $A$ and $B$ such that $B=f(A)$. Our intent is to understand which parameters of the model come into play in the numerical preservation of the structure and to extend the analysis to functions with singularities.
      
      In Section \ref{sec:off-diagonal} we see how the 
      integral definition of a matrix function
      enables us to study the structure of the off-diagonal blocks in $f(A)$. In Section~\ref{sec:QR} we develop the analysis of the singular values of structured outer products and we derive bounds for the off-diagonal singular values of matrix functions. 
     
      In Section~\ref{sec:poles} we adapt the approach to treat functions with 
      singularities. 
      
      The key role is played by an extension of the Dunford-Cauchy
      formula to the case in which some singularities lie inside the
      contour of integration.  In Section~\ref{sec:computational} we
      comment on computational aspects and we perform some experiments
      for validating the theoretical results, while in
      Section~\ref{sec:concludingremarks} we give some concluding
      remarks.
      
      \subsection{Definitions of matrix function}
      In \cite{high} ---which we indicate as a reference for
      this topic--- the author focuses on three equivalent definitions of matrix function. For our purposes we recall only two of them: one based on the Jordan canonical form of the argument and the other which is a generalization of the Cauchy integral formula.
      \begin{definition}\label{def:matrix-func1}
      Let $A\in\mathbb C^{m\times m}$  and $f(z)$ be a  function holomorphic in a set containing the spectrum of $A$. Indicating with $J=\hbox{diag}(J_1,\ldots,J_p)=V^{-1}AV$ the Jordan canonical form of $A$, we define $f(A):=V\cdot f(J)\cdot V^{-1}=V\cdot \diag(f(J_k))\cdot V^{-1}$ where $J_k$ is an 
      $m_k \times m_k$ Jordan block and
      \[
      \quad J_k=
      \begin{bmatrix}
      \lambda_k&1\\
      &\ddots&\ddots\\
      &&\ddots&1\\
      &&&\lambda_k
      \end{bmatrix},\quad  f(J_k)=\begin{bmatrix}
      f(\lambda_k)&f'(\lambda_k)&\dots&\frac{f^{(m_k-1)}(\lambda_k)}{(m_k-1)!}\\
      &\ddots&\ddots&\vdots\\
      &&\ddots&f'(\lambda_k)\\
      &&&f(\lambda_k)
      \end{bmatrix}.
      \]
      \end{definition}
      \begin{definition}[Dunford-Cauchy integral formula]   
             Let $f(z)$ be a holomorphic function in $\mathcal D\subseteq\mathbb C$ and $A\in\mathbb C^{m\times m}$ be a matrix whose spectrum is contained in $\Omega\subset\mathcal D$. Then we define
             \begin{equation}\label{cauchyformula}
             f(A):=
             \frac{1}{2\pi i}\int_{\partial\Omega}(zI-A)^{-1}f(z)dz.
             \end{equation}
             The matrix-valued function $\mathfrak R(z):=(zI-A)^{-1}$ is called \emph{resolvent}.
             \end{definition}
             Suppose that the spectrum of $A$ is contained in a disc $\Omega=B(z_0,r):=\{|z-z_0|<r\}$ where the function is holomorphic. Then, it is possible to write $f(A)$ as an integral \eqref{cauchyformula} along $S^1:=\partial B(0,1)$ for a matrix with spectral radius less than $1$. In fact,
             \begin{align*}
                     \frac{1}{2\pi i}\int_{\{|z-z_0|=r\}}(zI-A)^{-1}f(z)dz\quad=\quad \frac{1}{2\pi i}\int_{S^1}(wI-\tilde A)^{-1}f(rw+z_0)dw
                    \end{align*}
                    where $\tilde A=r^{-1}(A-z_0I)$ has the spectrum contained in $B(0,1)$. Given the above remark it is not restrictive to consider only the case of $A$ having spectral radius less than $1$.
                    
                    \begin{remark} \label{rem:spectralradiusD} In the
                      following we will often require, besides the non
                      singularity of $(zI-A)$, also that $(zI-D)$ is
                      invertible along the path of integration for any
                      trailing diagonal block $D$. This is not
                      restrictive since --- given a sufficiently large
                      domain of analyticity for $f$ --- one can choose
                      $r$ large enough which guarantees this
                      property. As an example, any $r$ such that
                      $r\geq\norm {A}$ is a valid choice for any
                      induced norm.
                    \end{remark}
   \section{Off-diagonal analysis of $f(A)$}
   \label{sec:off-diagonal}
   
   The study of the decay of the off-diagonal singular values
   has been investigated by \cite{chandrasekaran2010numerical}
   concerning the block Gaussian elimination on
   certain classes of quasiseparable matrices; in \cite{netna,qcr-krylov} the authors have proved
   fast decay properties that have been used to show the numerical quasiseparable preservation in the cyclic reduction \cite{buzbee,hockney,bim:book,SMC,CR}. 
   
   The aim of this section is characterizing the
   structure of the off-diagonal blocks by means
   of the integral definition of $f(A)$. 
   
   \subsection{Structure of an off-diagonal block}
   
   Consider the Dunford-Cauchy integral formula \eqref{cauchyformula} in the case $\partial\Omega=S^1$ and $A$ with the spectrum strictly
   contained in the unit disc. 
   In this case 
   the spectral radius of $A$ is less than $1$ and we can expand the resolvent as $(zI-A)^{-1}=\sum_{n\geq 0}z^{-(n+1)}A^n$.
          
           Applying component-wise the residue theorem we find that the result of the  integral in \eqref{cauchyformula} coincides with the coefficient of degree $-1$ in the Laurent expansion of $(zI-A)^{-1}f(z)$.       
          Thus, examining the Laurent expansion of an off-diagonal block, we can derive a formula for 
          the corresponding block in $f(A)$.         
          Partitioning $A$ as follows
          \[
          A=
          \begin{bmatrix}
          \bar A&\bar B\\
          \bar C&\bar D
          \end{bmatrix}\quad\Rightarrow\quad 
         \mathfrak R(z)=
          \begin{bmatrix}
                 zI-\bar A&-\bar B\\
                 -\bar C&zI-\bar D
                 \end{bmatrix}^{-1}
          \]
          and supposing that the spectral radius of $\bar D$ is less than $1$ (which is not restrictive thanks to
          Remark~\ref{rem:spectralradiusD}) we get
          \[
         \mathfrak R(z) =\begin{bmatrix}
                 S_{zI-\bar D}^{-1}&*\\
                 (zI-\bar D)^{-1}\bar CS_{zI-\bar D}^{-1}&*
                 \end{bmatrix}, 
          \] 
          where $S_{zI-\bar D} =zI-\bar A-\bar B(zI-\bar D)^{-1}\bar C$ is the Schur complement of the bottom right block and  $*$ denotes blocks which are not relevant for our analysis. 
           We can write the Laurent expansion of the two inverse matrices:
          \[
          (zI-\bar D)^{-1}=\sum_{j\geq 0}z^{-(j+1)}\bar D^j,\qquad S_{zI-\bar D}^{-1} = 
          \begin{bmatrix}
            I & 0 
          \end{bmatrix} \cdot \left(\sum_{j\geq 0}z^{-(j+1)}A^j\right)\cdot \begin{bmatrix}
            I \\ 0 
          \end{bmatrix},
          \]
          where for deriving the expansion of $S_{zI-\bar D}^{-1}$ we used that it corresponds to the upper left block in $\mathfrak R(z)$.
          
          Let  $f(z)=\sum_{n\geq 0}a_nz^n$ be the Laurent expansion of $f$ in $S^1$ and let $\mathfrak R(z)\cdot f(z):=\begin{bmatrix}
                 *&*\\
                 G(z)&*
                 \end{bmatrix}$, then  
          \begin{equation}\label{cbar}
          G(z)=\sum_{n\geq 0}a_n\sum_{j\geq 0}\bar D^j\bar C\cdot [I\quad 0]\cdot \sum_{s\geq 0}A^sz^{n-j-s-2}\cdot [I\quad 0]^t.
          \end{equation}
          Exploiting this relation we can prove the following.
                                            \begin{lemma}\label{lem:tech}
                                            Let $A=
          \begin{bmatrix}
          \bar A&\bar B\\
          \bar C&\bar D
          \end{bmatrix}$ be a square matrix with square diagonal blocks, $\bar C=uv^t$ and suppose that the spectrum of $A$ and $\bar D$ is contained in $B(0, 1)$. Consider $f(z)=\sum\limits_{n\geq 0}a_nz^n$ for $|z|\leq 1$ and let   $f(A)=\begin{bmatrix}
          *&*\\
          \tilde C&*
          \end{bmatrix}$ be partitioned according to $A$. Then
          \begin{align*}
          \tilde C&=\sum_{n\geq 1}a_n \left[
                                                                    u\ \vline\ \bar D\cdot u\ \vline\ \dots\ \vline\ \bar D^{n-1}\cdot u
                                                                    \right]\cdot
                                                                    \left[
                                                                        (A^t)^{n-1}\tilde v\ \vline\ \dots\ \vline\ A^t\tilde v\ \vline\ \tilde v
                                                                        \right]^t[ I \ 0 ]^t
              \end{align*}
              with $\tilde v=[I\ 0]^tv$.
   \end{lemma}  
                \begin{proof}
                By the Dunford-Cauchy formula, the subdiagonal block $\tilde C$ is equal to $\int_{S^1}G(z)dz$. By means of the residue theorem we can write  the latter as
                the coefficient of degree $-1$ in \eqref{cbar}, that is
                \[
                \tilde C = \sum_{n\geq 1}a_n\sum_{j=0}^{n-1}\bar D^juv^t\cdot [I\ 0]A^{n-j-1}[I\ 0]^t=\sum_{n\geq 1}a_n\sum_{j=0}^{n-1}\bar D^ju\tilde v^tA^{n-j-1}[I\ 0]^t, 
                \] which is in the sought form. 
                \end{proof} 
                
                \begin{remark}\label{rem:outer}
                The expression that we obtained for $\tilde C$  in the previous Lemma is a sum of outer products of vectors of the form $\bar D^ju$ with $(A^t)^{n-j-1}\tilde v$, where the spectral radii of $A$ and $\bar D$ are both less than $1$. 
                This implies that the addends 
                become negligible for a sufficiently large $n$. So, in order to derive bounds for the singular values, we will focus on the truncated sum
                \begin{equation}
                \label{kryl_negl}
                \sum_{n= 1}^sa_n \left[
             u\ \vline\ \bar D\cdot u\ \vline\ \dots\ \vline\ \bar D^{n-1}\cdot u
             \right]\cdot
             \left[
                 (A^t)^{n-1}\tilde v\ \vline\ \dots\ \vline\ A^t\tilde v\ \vline\ \tilde v
                 \right]^t[ I \ 0 ]^t 
                 \end{equation}
                 which can be rewritten as:
                 \begin{equation}
                 \label{kyl/horn}
                 \left[
                                              u\ \vline\ \bar D\cdot u\ \vline\ \dots\ \vline\ \bar D^{s-1}\cdot u
                                              \right]\cdot
                                              \left[
                                                  \sum_{n=0}^{s-1}a_{n+1}(A^t)^n\tilde v\ \vline\ \dots\ \vline\ (a_s A^t+a_{s-1}I)\tilde v\ \vline\  a_s \tilde v
                                                  \right]^t[ I \ 0 ]^t.
    \end{equation}
    The columns of the left factor span the Krylov subspace $\mathcal K_n(\bar D,u):=\operatorname{Span}\{u,\bar D u, \dots,\bar D^{n-1}u\}$.
    
    Let $p(z):=\sum_{n=0}^{s-1}a_{n+1} z^n$.
    Looking closely at the columns of the right factor in \eqref{kyl/horn} we can see that they correspond to the so called Horner shifts (which are the intermediate results obtained while evaluating a polynomial using the Horner rule \cite{henrici}) of $p(A^t)\tilde v$. In the following we will refer to the patterns in the factors of \eqref{kyl/horn} as \emph{Krylov} and \emph{Horner} matrices, respectively.
    \end{remark}
     \section{Outer products, QR factorization and singular values}\label{sec:QR}
     The problem of estimating the numerical rank of an outer product is addressed for example in \cite{netna},
       where the authors estimate
        the singular values
        of a matrix $X = \sum_{i = 1}^{n} u_i v_i^*$ ---where the superscript $*$ stands for the usual complex conjugate transposition--- exploiting the exponential decay in the norms of  the rank $1$ addends.   
        However, such an estimate is sharp only when the vectors
        $u_i$ and $v_i$ are orthogonal. In general, 
        the singular values of $X$ decay quickly also when 
        the vectors $u_i$ and/or $v_i$ tend to become parallel as $i$ increases.  
        For this reason, in this work we   rephrase the 
        expression $X = \sum_{i = 1}^n u_i v_i^*$ as 
        $X = \sum_{i = 1}^m \tilde u_i \tilde v_i^*$ where
        $\tilde u_i$ and $\tilde v_i$ are chosen as ``orthogonal
        as possible''.   
        To this aim we study the QR decomposition of the  matrices
        \[
          U = \begin{bmatrix}
            \\
            u_1 & u_2 & \cdots & u_n \\
            \\
          \end{bmatrix},\qquad
           V = \begin{bmatrix}
                \\
                v_1 & v_2 & \cdots & v_n  \\
                \\
              \end{bmatrix}.
        \]
        
        We indicate their QR decompositions 
        as $U = Q_UR_U$ and $V = Q_VR_V$ where
        $Q_U,Q_V$ are $m \times m$ and $R_U,R_V$ have $m$ rows and $n$
        columns.   
        
       This section is divided into five parts. In the first we study the element-wise decay in the QR factorization of Krylov matrices. In the second we show how to handle the case in which the matrix $A$ 
       is not diagonalizable. In the third we study the same properties for Horner matrices. In Section~\ref{sec:krylovhornersingdecay} we show that the singular values of a Krylov/Horner outer product inherit the decay. Finally, in Section~\ref{sec:singdecay} we derive bounds for the off-diagonal singular values of $f(A)$. 
      \subsection{Decay in the entries of the $R$ factor for Krylov matrices}
      
      In this section we show how to exploit the relation
      between Krylov subspaces and polynomial approximation 
      \cite{saad2003iterative}. 
      More precisely, we relate the decay in the matrix $R$ with the convergence of a minimax polynomial approximation problem in a subset of the complex plane.   
      
      The rate of convergence of the latter problem depends on the geometry
      of the spectrum of $A$. In particular, for every
      compact connected subset of $\mathbb C$ that contains the spectrum
    we obtain an exponent for the decay depending on 
    its logarithmic capacity \cite{landkof,markushevich2005theory}.
   	 
     In order to simplify the exposition, in this section we will assume that 
     the matrix $A$ is diagonalizable. However, this is not strictly required
     and in the next subsection we show
     how to relax this hypothesis. 
     
     Our approach is inspired by the one of Benzi and Boito in \cite{benzi2013decay,benzi2014decay}, where the authors proved the numerical preservation of sparsity patterns in matrix functions. For a classic reference of the complex analysis behind the next
     definitions and theorems we refer to \cite{markushevich2005theory}. 
     
          \begin{definition}[Logarithmic capacity]
            Let $F \subseteq \mathbb C$ be a nonempty, compact and 
            connected set, and 
            denote with $G_\infty$ the connected component of the complement containing
            the point at the infinity. Since $G_\infty$ is simply connected, 
            in view of the Riemann Mapping Theorem we know that there exists
            a conformal map $\Phi(z)$ which maps $G_\infty$ to the complement of
            a disc. If we impose the normalization conditions
            \[
              \Phi(\infty) = \infty, \qquad 
              \lim_{z \to \infty} \frac{\Phi(z)}{z} = 1
            \]
            then this disc is uniquely determined. We say that its radius
            $\rho$ is the \emph{logarithmic capacity} of $F$ and we write $\lc(F)=\rho$.         
            Let $\Psi=\Phi^{-1}$, for every $R>\rho$ we indicate with $C_R$ the image under $\Psi$ of the circle $\{|z|=R\}$. 
          \end{definition}  
           
         The logarithmic capacity is strictly
         related to the following well-known result of polynomial approximation in the complex plane.
         
           \begin{lemma}[Corollary~2.2 in \cite{ellacott1983computation}]\label{lem:minimax-approx}
                 Let $F$ be a Jordan region whose boundary is of finite total rotation $\mathcal V$ and of logarithmic capacity $\rho$. If $f(z)$ is an analytic function  on $\mathbb{C}$ then $\forall r>\rho$
                 and any integer $i\geq 0$ there exists a polynomial $p_i(z)$ of degree at most $i$  such that 
                 \[
                 \lVert f(z)-p_i(z)\rVert_{\infty, F}\leq \frac{M(r)\mathcal V}{\pi(1-\frac{\rho}{r})}\left(\frac{\rho}{r}\right)^{i+1}.
                 \]
                 with $M(r):=\max_{C_r} |f(z)|$.
                 \end{lemma}
          
          In order to exploit Lemma~\ref{lem:minimax-approx} in our framework
          we need to introduce some new constants related to the geometry
          of the set $F$. 
           
          \begin{definition}
                        Given $F \subseteq \mathbb C$ compact, connected with $\lc(F)=\rho\in(0,1)$ we indicate with $R_F$ the quantity
                        \[
                        R_F:=\sup\{R>\rho:\ C_R \text{ is strictly contained in the unit circle}\}.
                        \]
                        \end{definition}   
          \begin{definition}
            We say that $F\subset \mathbb C$ is {\em enclosed} by
            $(\rho,R_F,\mathcal V_F)$ if $\exists F'$ Jordan region
            whose boundary has finite total rotation\footnote{%
              See \cite[Section 2, p. 577]{ellacott1983computation} for the definition of
              total rotation.}
            $\mathcal V_F$,
            $\lc(F')=\rho$, $R_F=R_{F'}$ and $F\subseteq F'$.
          \end{definition}
          
          \begin{definition}
          We say that $A\in\mathbb C^{m\times m}$ is {\em enclosed} by $(\rho,R_A,\mathcal V_A)$ if the set of its eigenvalues is enclosed by $(\rho,R_A,\mathcal V_A)$.
          \end{definition}
          
      \begin{definition}
      Let $J$ be the Jordan canonical form of $A\in\mathbb C^{m\times m}$. Let $\mathbb V :=\{V\in \mathbb{C}^{m\times m}:\ V^{-1}AV=J\}$.  We define the quantity
      \[
      \kappa_{s}(A):=\inf_{V\in\mathbb V}\parallel V\parallel_2 \parallel V^{-1}\parallel_2.
      \]
      \end{definition}
          
          We can now proceed to study the $R$ factor of a Krylov matrix.
          
          \begin{theorem} \label{thm:krylov-decay}
            Let  $A\in\mathbb{C}^{m\times m}$   be a diagonalizable matrix enclosed by $(\rho,R_A,\mathcal V_A)$, $\rho\in(0,1)$ and $b\in\mathbb{C}^m$. Moreover,
            let
            $U$ be the matrix whose columns span the $n$-th Krylov 
            subspace $\mathcal{K}_n(A,b)$:
            \[
              U = \left[ \begin{array}{c|c|ccc}
               b & Ab & A^2 b & \ldots & A^{n-1} b \\
              \end{array} \right]. 
            \]
            Then $\forall r\in (\rho,R_A)$  the entries of the R factor in the QR decomposition of $U$ satisfy
            \[
            |R_{ij}| \leq c(r)\cdot \kappa_{s}(A)\cdot \left(\frac{\rho}{r}\right)^{i} \delta^j 
            \]
             where $\delta=\max_{z\in C_r}|z|$ and  $c(r)=\frac{\mathcal V_A}{\delta\pi(1-\frac{\rho}{r})}\cdot \norm {b}_2$.
          \end{theorem}

          \begin{proof}
             Let $QR = U$ be the QR factorization of $U$ and $V^{-1}AV=D$ the spectral decomposition of $A$. Notice that the quantity $\norm{R_{i+1:j,j}}_2$
             is equal to the norm of the projection
             of $u_j$ on the orthogonal to the 
             space spanned by 
             the first $i$ columns of $U$, that
             is $\mathcal K_{i}(A, b)^\perp$. 
             It is well-known
            that the Krylov subspace $\mathcal{K}_i(A,b)$ contains
            all the vectors of the form $p(A) b$ where $p$ has 
            degree at most $i-1$. 
            In particular, we have:
            \begin{align*}
              |R_{i+1,j}| \leq \norm{R_{i+1:j,j}}_2 \leq 
              \min_{deg(p)=i-1}\lVert p(A) b - u_{j} \rVert_2&\leq \min_{deg(p)=i-1}\lVert p(D)  - D^{j-1} \rVert_2 \lVert V^{-1}\rVert_2 \lVert  V\rVert_2\lVert b\rVert_2\\ &\leq \frac{M(r)\mathcal V_A}{\pi(1-\frac{\rho}{r})}\left(\frac{\rho}{r}\right)^i \kappa_{s}(A) \norm{b}_2,
            \end{align*}
            where $M(r)=\max_{C_r} |z|^{j-1}=\delta^{j-1}$.
          \end{proof}

          \subsection{Non diagonalizable case}\label{subsec:nondiag}
          The diagonalizability hypothesis can be relaxed using different  strategies. We first 
          propose to rely on a well-known result by Crouzeix \cite{crouzeix2007numerical} based on the numerical range. Then, we 
          discuss another approach consisting
          in estimating the minimax  approximation error on the Jordan canonical form.
          \subsubsection{Numerical range}
          In the spirit of the results found in \cite{benzi2014decay}, we 
          can give an alternative formulation that avoids the requirement
          of diagonalizability. The price to pay consists in having to estimate
          the minimax error bound on a set larger than the spectrum.        
          To be precise, we need to consider the numerical range of the matrix $A$. 
          
          \begin{definition}
            Let $A$ be a matrix in $\mathbb C^{m \times m}$. We define 
            its numerical range $\mathcal W(A)$ as the set 
            \[
              \mathcal W(A) = \left\{ x^* A x \ | \ 
              x \in \mathbb C^m, \ \ 
              \lVert x \rVert_2 = 1 \right\} \subseteq \mathbb C. 
            \]
          \end{definition}
          
          The numerical range is a compact convex subset of $\mathbb C$ which
          contains the eigenvalues of $A$. When $A$ is normal 
          $\mathcal W(A)$ is exactly the convex hull of the
          eigenvalues of $A$. Moreover, it has a strict connection with the evaluation of matrix functions \cite{crouzeix2007numerical}, 
          which is described by the following result. 
          
          \begin{theorem}[Crouzeix]\label{thm:crouzeix}
          There is a universal constant $2\leq \mathcal C\leq 11.08$ such that, given $A\in\mathbb{C}^{m\times m}$,
          and a continuous function $g(z)$ on 
          $\mathcal W(A)$, analytic in its interior,  the following inequality holds:
          \[
          \lVert g(A)\rVert_2\leq \mathcal C \cdot \lVert g(z)\rVert_{\infty, \mathcal W(A)}.
          \]
          \end{theorem}    
          Whenever the numerical range
          $\mathcal W(A)$ has a logarithmic capacity 
          smaller than $1$ it is possible to extend 
          Theorem~\ref{thm:krylov-decay}.
          \begin{corollary}
          Let  $A\in\mathbb{C}^{m\times m}$ be such that the field of values
          $\mathcal W(A)$ is enclosed by $(\rho,R_{\mathcal{W(A)}},\mathcal V_{\mathcal W(A)})$, $\rho\in(0,1)$ and $b\in\mathbb{C}^m$. Moreover,
                   let
                   $U$ be the matrix whose columns span the $n$-th Krylov 
                   subspace $\mathcal{K}_n(A,b)$:
                   \[
                     U = \left[ \begin{array}{c|c|c|c|c}
                      b & Ab & A^2 b & \ldots & A^{n-1} b \\
                     \end{array} \right]. 
                   \]
                   Then $\forall r\in (\rho,R_{W(A)})$  the entries of the R factor in the QR decomposition of $U$ satisfy
                   \[
                   |R_{ij}| \leq c(r)\cdot \left(\frac{\rho}{r}\right)^{i} \delta^j
                   \]
                    where $\delta=\max_{z\in C_r}|z|$ and $c(r)=\frac{\mathcal C\cdot\mathcal V_{\mathcal W(A)}}{\delta\pi(1-\frac{\rho}{r})}\cdot \norm{b}_2$.
          \end{corollary}
          \begin{proof} 
          Follow the same steps in the 
          proof of Theorem~\ref{thm:krylov-decay}
          employing Theorem~\ref{thm:crouzeix} to bound $R_{ij}$.
          \end{proof}
         \subsubsection{Jordan canonical form}
            
          An alternative to the above approach is to rely on the Jordan 
          canonical form in place of the eigendecomposition. 
          More precisely, we can always write any matrix
          $A$ as $A = V^{-1} J V$ with $J$ being block
          diagonal with bidiagonal blocks (the so-called Jordan blocks).
          This implies that
          the evaluation of $f(J)$ is block diagonal
          with blocks $f(J_t)$ where $f(J_t)$ have the
          following form: 
          \[
             \quad J_t=
             \begin{bmatrix}
             \lambda_t&1\\
             &\ddots&\ddots\\
             &&\ddots&1\\
             &&&\lambda_t
             \end{bmatrix}
             \in\mathbb C^{m_t\times m_t},\quad f(J_t)=\begin{bmatrix}
             f(\lambda_t)&f'(\lambda_t)&\dots&\frac{f^{(m_t-1)}(\lambda_t)}{(m_t-1)!}\\
             &\ddots&\ddots&\vdots\\
             &&\ddots&f'(\lambda_t)\\
             &&&f(\lambda_t)
             \end{bmatrix}.
          \]
          
          We can evaluate the matrix function $f(A)$
          by $f(A) = V^{-1} f(J) V$. One can
          estimate the norm $\norm{R_{i+1:j,j}}_2$ 
          as in the proof of Theorem~\ref{thm:krylov-decay}: 
           \begin{equation}\label{Rjordan}
           |R_{i+1,j}| \leq \norm{R_{i+1:j,j}}_2 \leq \min_{deg(p)=i-1}\lVert p(A) b - u_{j} \rVert_2\leq \min_{deg(p)=i-1}\lVert p(J)  - J^{j-1} \rVert_2 \cdot \kappa_s(A)_2 \cdot \lVert b\rVert_2
         \end{equation}
         where $p(J)=\diag(p(J_t))$, $J^j=\diag(J_t^j)$ and
         \begin{equation}\label{jordan}
         p(J_t)-J_t^j=\begin{bmatrix}
                  p(\lambda_t)-\lambda_t^j&p'(\lambda_t)- j\lambda_t^{j-1}&\dots&\frac{p^{(m_t-1)}(\lambda_t)}{(m_t-1)!}-\frac{j!}{(j-m_t)!(m_t-1)!}\lambda_h^{j-m_t}\\
                  &\ddots&\ddots&\vdots\\
                  &&\ddots&p'(\lambda_t)- j\lambda_t^{j-1}\\
                  &&&p(\lambda_t)-\lambda_t^j
                  \end{bmatrix}.
         \end{equation}
         We can rephrase \eqref{Rjordan} as a problem of simultaneous approximation of a function and its derivatives
      \begin{lemma}
      Let $\mathcal S$ be a simply connected subset of the complex plane  and suppose that $\exists z_0\in\mathcal S$ such that each element of $\mathcal S$ can be connected to $z_0$ with a path of length less than $1$.
        Let $p(z)$ be a degree $i$
        polynomial approximating the holomorphic function
        $f'(z)$ in $\mathcal S$, such that $|f'(z) - p(z)| \leq \epsilon$ in $\mathcal S$. 
        Then there exists a polynomial $q(z)$
        of degree $i + 1$ with
        $q'(z) = p(z)$ such that 
        \[
          |q(z) - f(z)| \leq \epsilon \qquad z \in \mathcal S,
        \]
      \end{lemma}
      
      \begin{proof}
        Define $q(z)$ as follows:
        \[
          q(z) = f(z_0) + \int_{\gamma} p(z), \qquad 
          \gamma \text{ any path connecting } z_0 \text{ and } z. 
        \]
        The above definition uniquely determines 
        $q(z)$, and we know that it is a polynomial
        of degree $i + 1$. 
        Given $z \in \mathcal S$ choose $\gamma$ a
        path connecting $z_0$ to $z$ with length less than $1$, we have:
        \[
          |f(z) - q(z)| = |f(z_0) + \int_{\gamma} f'(z) - 
            f(z_0) - \int_{\gamma} p(z)| \leq 
            \int_{\gamma} |f'(z) - p(z)| \leq \epsilon. 
        \]
      \end{proof}
      
       If $m_{t'}$ is the maximum size among all the Jordan blocks we can find a minimax approximating polynomial
       for the $m_{t'}$ derivative of $z^j$. The
       above Lemma guarantees that, with the latter choice,
       the matrix \eqref{jordan} has the $(i,j)$-th entry bounded in modulus by $\frac{\epsilon}{(j-i)!}$ when $j\geq i$. 
      An easy computation shows that both the $1$ and $\infty$ norms of 
      \[  
          T=\epsilon\begin{bmatrix}
                   1&1&\frac{1}{2!}&\dots&\frac{1}{(m_{t'}-1)!}\\
                    &\ddots&\ddots&\ddots&\vdots\\
                    &&\ddots&\ddots&\frac{1}{2!}\\
                    &&&\ddots&1\\
                    &&&&1
                    \end{bmatrix}
      \]
      are bounded by $\epsilon e$, where $e$ is the Napier's constant. We then have $\norm{p(J)-J^k}_2\leq\norm{T}_2\leq \sqrt{\norm{T}_1\norm{T}_{\infty}}\leq \epsilon e$. Using this relation one can prove the next result by following the same steps as in the proof of Theorem~\ref{thm:krylov-decay}.
      \begin{theorem} 
              Let  $A\in\mathbb{C}^{m\times m}$, $b\in\mathbb{C}^m$ and $F$ be the convex hull of the spectrum of $A$. Suppose that $F \subseteq B(0,1)$
              is enclosed by $(\rho,R_F,\mathcal V_F)$, $\rho\in(0,1)$ and  indicate with $m_{t'}$ the size of the largest Jordan block of $A$.  Moreover,
              let
              $U$ be the matrix whose columns span the $n$-th Krylov 
              subspace $\mathcal{K}_n(A,b)$:
              \[
                U = \left[ \begin{array}{c|c|ccc}
                 b & Ab & A^2 b & \ldots & A^{n-1} b \\
                \end{array} \right]. 
              \]
              Then $\forall r\in (\rho,R_F)$  the entries of the R factor in the QR decomposition of $U$ satisfy
                       \[|R_{ij}| \leq c(r)\cdot \kappa_{s}(A)\cdot \left(\frac{\rho}{r}\right)^{i-(m_{t'}-1)} \delta^j, \] where $\delta=\max_{z\in C_r}|z|$and $c(r)=\frac{e\cdot\mathcal V_{F}}{\delta\pi(1-\frac{\rho}{r})}\cdot \norm{b}_2$.
            \end{theorem}
            \subsection{Decay in the entries of the $R$ factor for Horner matrices}
            Here, we show that the two-way decay in the $R$ factor is shared by the right one in \eqref{kyl/horn}, which we have identified as 
            Horner matrix. 
                                           
            \begin{theorem}\label{thm:horner-decay}         
      Let  $A\in\mathbb{C}^{m\times m}$   be a diagonalizable matrix enclosed by $(\rho,R_A,\mathcal V_A)$, $\rho\in(0,1)$ and $b\in\mathbb{C}^m$. 
              Moreover let
               $U$ be the matrix:
              \[
                   U =\left[                              a_s b\  \vline\ (a_s A+a_{s-1}I)b\ \vline\ \dots\ \vline \ \sum_{j=0}^{s-1}a_{j+1}A^j b
                                              \right]
                                              \]
               where the finite sequence $\{a_j\}_{j=1,\dots, s}$ verifies
               \[
               |a_j|\leq  \hat \gamma \cdot \hat \rho^{j},\quad \hat \gamma>0, \quad \hat \rho\in(0,1),\qquad  j=1,\dots, s.
               \]
                Then  the R factor in the QR decomposition of $U$ is entry-wise bounded by
                                 \[|R_{ij}| \leq c\cdot \kappa_{s}(A)\cdot \left(\frac{\rho}{R_A}\right)^{i} \hat \rho^{i+(s-j)} \]
                                 where $c=\frac{\hat \rho\hat{\gamma}\mathcal V_A}{\pi(1-\hat\rho)(1-\frac{\rho}{R_A})}\norm{b}_2$.
       \end{theorem}
      \begin{proof}
      Here we assume that $a_{s} \neq 0$. This is not restrictive because
      if $j < s$ is the largest $j$ such that $a_{j'} = 0$ for any $j' > j$
      the first $s - j$ columns of $U$ are zero, and can be ignored.
      Observe that the $j$-th column of $U$ is of the form $q(A) b$ where $q$ is the polynomial defined by the coefficients $a_j$ in reversed order, i.e., 
      \[
        q(x) := \sum\limits_{n=0}^{j-1}a_{s-j+1+n}x^n.
      \]
      The subspace spanned by the first $i$ columns of $U$ contains all the vectors of the form $p(A) b$ where $p$ is a polynomial of degree at most $i-1$.
      With the same argument used for proving Theorem~\ref{thm:krylov-decay}  we  can bound the entries of $R$ in this way
      \[
      |R_{ij}|\leq \min_{deg(p)=i-1}\norm{p(D)-\sum_{n=0}^{j-1}a_{s-j+1+n}D^{n}}_2\cdot  \kappa_s(A)\cdot \norm{b}_2.
      \]
      Moreover
      \begin{align*}
      \min_{deg(p)=i-1}\norm{p(D)-\sum_{n=0}^{j-1}a_{s-j+1+n}D^n}_2&=  \min_{deg(p)=i-1}\norm{p(D)-\sum_{n=i}^{j-1}a_{s-j+1+n}D^n}_2\\
      &\leq \sum_{n=i}^{j-1}|a_{s-j+1+n}|\min_{deg(p)=i-1}\norm{p(D)-D^n}_2\\
      &\leq \sum_{n=i}^{j-1}\hat \gamma\hat \rho^{s-j+1+n}\min_{deg(p)=i-1}\norm{p(D)-D^n}_2\\
      &\underbrace{\leq}_{\text{Lemma~\ref{lem:minimax-approx}}}\sum_{n=i}^{j-1}\hat \gamma\hat \rho^{s-j+1+n} \frac{\mathcal V_A}{\pi(1-\frac{\rho}{R_A})}\left(\frac{\rho}{R_A}\right)^i\\
      &\leq \frac{\hat \rho\hat{\gamma}\mathcal V_A}{\pi(1-\hat\rho)(1-\frac{\rho}{R_A})}  \hat\rho^{s-j+i}\left(\frac{\rho}{R_A}\right)^i,
      \end{align*}
      where we used Lemma~\ref{lem:minimax-approx} with $r=R_A$. 
       \end{proof}   
       \begin{remark}
       In view of the above arguments we can rephrase 
       Theorem~\ref{thm:krylov-decay} for non diagonalizable matrices. We obtain similar statements involving $\lc(\mathcal W(A))$ in place of $ \lc(A)$ or with a shifted column decay. The same technique can be used to generalize the results of the next sections. 
       The proofs and statements are analogous to the diagonalizable case. Therefore, we do not report them.
                    \end{remark}
     \subsection{Decay in the singular values of Krylov/Horner outer products}   \label{sec:krylovhornersingdecay}
        
        \subsubsection{Some preliminaries}
        In what follows, we indicate with $\Pi_m$ the counter identity of order $m$:
        \[
        \Pi_m:=\begin{bmatrix}
        &&1\\
        &\iddots\\
        1
        \end{bmatrix}\in\mathbb R^{m\times m},
        \]
        which is the matrix which flips the columns.
        
        Due to technical reasons, we also need to introduce the following quantity.
        \begin{definition}
Given $A\in\mathbb C^{m\times m}$ enclosed by $(\rho,R_{A},\mathcal V_{A})$ and a parameter $R\in\mathbb R^+$ we define 
        \[
        \Lambda(\rho,R_{A},\mathcal V_{A},R):=\frac{\mathcal V_{A}^2}{\pi^2(R-1)(1-\frac{\rho}{R_{A}})\sqrt{1-(\frac{\rho}{R R_A})^2}}\cdot \min_{\rho< r<R_A}\frac{1}{\delta(r)(1-\delta(r)^2)(\frac{r}{\rho}-1)\sqrt{(1-\frac{\rho^2}{r^2})}},
        \]
        where $\delta(r):=\max\{\frac{1}{R},\max_{C_r}|z|\}$.
                \end{definition}
        \subsubsection{The estimates}
        Now, we have all the ingredients for studying  the singular values of Krylov/Horner outer products. 
        For simplicity we state a result in the diagonalizable case, but we highlight that it is easy to recover analogous estimates for the general framework employing the techniques of Section~\ref{subsec:nondiag}.
       \begin{theorem}\label{thm:kryl-horn-decay}
       Let $b_1\in\mathbb C^{m}$, $b_2\in\mathbb C^{n}$ and  $A_1\in\mathbb C^{m\times m},A_2\in\mathbb C^{n\times n}$ be two diagonalizable matrices enclosed by $(\rho,R_{A},\mathcal V_{A})$ with $\rho\in(0,1)$.
       Then for any finite sequence $\{a_j\}_{j=1,\dots, s}$ which verifies
   \[
      |a_j|\leq  \hat \gamma\cdot  R^{-j}, \quad R>1,\quad  j\in\{1,\dots, s\}, 
   \]
      the singular values of
   \begin{equation}\label{horn-outer}
         	X=\left[ \begin{array}{c|c|cc}
      b_1 & A_1b_1 &  \ldots & A_1^{s-1} b_1 \\
     \end{array} \right]\cdot
   \left[  \sum_{j=0}^{s-1}a_{j+1}A_2^j b_2                            \ \vline\ \dots\ \vline\ (a_s A_2+a_{s-1}I)b_2\  \vline \ 
                        a_s b_2 \right]^t                                                                                     \end{equation}
                        can be bounded by
        \begin{align*}
        \sigma_{l}(X) &\leq  \gamma\cdot e^{-(\alpha+\alpha') (l+1)},\qquad \alpha=\log\left(\frac{R_A}{\rho}\right),\qquad   \alpha'=\log\left(R\right),
        \end{align*}
        where $ \gamma :=\hat{\gamma}\cdot\kappa_s(A_1) \kappa_s(A_2)\norm {b_1}_2\norm{b_2}_2\cdot \Lambda(\rho,R_{A},\mathcal V_{A},R)$.
        \end{theorem}
       \begin{proof}
      Consider the matrices $U$ and $V$ defined as follows: 
      \[
        U =
          \left[ \begin{array}{c|c|cc}
        b_1 & A_1b_1 &  \ldots & A_1^{s-1} b_1 \\
        \end{array} \right], \quad 
        V = 
        \left[ a_s b_2  \ \vline\ (a_s A_2+a_{s-1}I)b_2\
        \vline\ \dots\  
        \vline \ \sum_{j=0}^{s-1}a_{j+1}A_2^j b_2 
        \right], 
      \]
      so that 
       we have $X = U \Pi_s V^t$ as in Equation~\eqref{horn-outer}. Moreover, let $(Q_U, R_U)$ and 
   $(Q_V, R_V)$  be the QR factorizations of $U$ and $V$ respectively. Applying Theorem~\ref{thm:krylov-decay} and Theorem~\ref{thm:horner-decay} we get that $\forall r\in(\rho,R_{A})$ 
   \[
   |R_{U,ij}| \leq c_1(r)\cdot e^{-\eta i - \beta j}\qquad \text{and}\qquad |R_{V,ij}| \leq c_2\cdot  e^{-(\alpha+\alpha') i - \beta (s-j)},
   \]
   with $\eta =\log\left(\frac{r}{\rho}\right)$,  $\beta=|\log(\delta)|$, $c_1(r)=\frac{\mathcal V_{A_1}}{\delta\pi(1-\frac{\rho}{r})}\cdot \kappa_s(A_1)\cdot \norm {b_1}_2$ and $c_2=\frac{\hat \rho\hat{\gamma}\mathcal V_{A_2}}{\pi(1-\hat\rho)(1-\frac{\rho}{R_{A}})}\kappa_s(A_2)\norm{b_2}_2$.
   
   In order to bound 
   the singular values of $X$ we look at those of $S = R_U \Pi_s R_V^*$. 
   The entry $(i,j)$ of $S$ is obtained as 
   the sum: \[
   S_{ij} = \sum_{h = 1}^s R_{U,ih} \cdot R_{V,j(s-h)}, \qquad 
   | R_{U,ih} \cdot R_{V,j(s-h)} | \leq 
   c\cdot e^{-\eta i-(\alpha+\alpha')j - 2\beta h},
   \] 
   where $c=c_1(r)\cdot c_2$.
   Summing
   all the bounds on the addends we obtain 
   \[
   |S_{ij}| \leq \frac{c}{1-e^{-2\beta }}  e^{-\eta i-(\alpha+\alpha') j}.
   \]
   We can estimate the $l$-th singular value
   by setting the first $l-1$ columns of $S$ to zero.
   Let $S_l$ be the matrix composed by the last $m - l+1$
   columns of $S$. Since this matrix can be seen as the residue of 
   a particular choice
   for a rank $l-1$ approximation of $S$ we have 
   $
    \sigma_{l}(S) \leq \norm{S_l}_2        
   $. The entries of $S_l$ satisfy
   the relation $(S_l)_{ij} \leq \tilde \gamma e^{-(\alpha+\alpha') l} e^{-\eta i-(\alpha+\alpha')(j-1))}$ where $\tilde \gamma=\frac{c}{1-e^{-2\beta}}$, so we obtain:
   \begin{align*}
     \norm*{\frac{e^{(\alpha+\alpha') l}}{\tilde \gamma} S_l}_F^2 &=  \sum_{i = 1}^{m-l} \sum_{j = 1}^n |\frac{e^{(\alpha+\alpha') l}}{\tilde \gamma}(S_k)_{i,j}|^2 
     \leq \frac{e^{-2\eta}}{(1-e^{-2\eta})(1-e^{(-\alpha+\alpha')})}. 
   \end{align*}
       Since $\norm{S_l}_2 \leq \norm{S_l}_F$ we 
       have $\sigma_{l}(S) \leq \frac{\tilde \gamma e^{-\eta}}{\sqrt{(1-e^{-2\eta})(1-e^{-2(\alpha+\alpha')})}} e^{-(\alpha+\alpha') l}=\gamma e^{-(\alpha+\alpha') l}$.
    \end{proof}
    Our final aim is to estimate the singular values of
    \eqref{kryl_negl} by estimating the singular values of one of its
    finite truncations \eqref{kyl/horn}. In order to justify that, we
    need to show that the addends in \eqref{kryl_negl} become
    negligible. Observe that the latter are outer products of two
    Krylov matrices in which the second factor appears in a reverse
    order. This means that the row-decay in its $R$ factor has an
    opposite direction. In the next result we see how this fact
    implies the negligibility.
     \begin{theorem} \label{thm:sdecay2}
    Let $U = Q_U R_U$ and 
    $V = Q_V R_V$  be QR factorizations of 
    $U \in \mathbb{C}^{m \times n}$ and
    $V \in \mathbb{C}^{m \times n}$. 
    Let 
    $\alpha, \beta$
    and $c$ be positive constants such
    that $|R_{U,ij}|,|R_{V,ij}| \leq c e^{-\alpha i - \beta j}$ for any $i,j$. Then the
    matrix $X = U\Pi_n V^*$
     has singular 
    values bounded
    by 
    \[
      \sigma_{l}(X) \leq  \gamma e^{-\alpha (l+1)}, 
      \qquad  \gamma := \frac{c^2ne^{-(n+1)\beta}}{(1 - e^{-2\alpha})} . 
    \]
   \end{theorem}
   
   \begin{proof}
   We can write $X = U\Pi_n V^* = 
   Q_U R_U \Pi_n R_V^* Q_V^*$, so its
   singular values coincide with the ones
   of $S = R_U \Pi_{n,m} R_V^*$. 
   The element in 
   position $(i,j)$ of $S$ is obtained as 
   the a sum \[
   S_{ij} = \sum_{l = 1}^n R_{U,il} \cdot R_{V,j(n-l)}, \qquad 
   | R_{U,il} \cdot R_{V,j(n+1-l)} | \leq 
   c^2 e^{-\alpha(i+j) - \beta (n+1)}
   \] according to our
   hypotheses. Since the bound on the 
   elements in the above summation is independent of $n$ 
   we can write
   $
   |S_{ij}| \leq c^2 n e^{-\beta (n+1)} e^{-\alpha(i+j)}
   $.
   The thesis can then be obtained by following
   the same procedure as in Theorem~\ref{thm:kryl-horn-decay}. 
   \end{proof}  
   \begin{remark}\label{negligible}
   Observe that the larger $n$ the closer the quantity $ne^{-\beta n}$ is to $0$. Therefore for sufficiently big $n$ the resulting matrix is  negligible.
   \end{remark}
      
    \subsection{Decay in the  off-diagonal singular values of $f(A)$}    \label{sec:singdecay}
    We start with a few technical results that will make some proofs  smoother.  
    \begin{lemma}\label{dyads}
    Let  $A^+=\sum_{j=0}^{+\infty} A_j$  with $A_j\in\mathbb{R}^{m\times n}$ matrices of rank $k$ and suppose that $\left\| A_j\right\|_2\leq \gamma e^{-\alpha |j|}$. Then 
    \[
     \sigma_l(A^+)\leq \frac{\gamma}{1-e^{-\alpha}}\cdot e^{-\alpha  \frac{l-k}{k}}.
    \]
    \end{lemma}
    \begin{proof}
    Note that $\sum\limits_{j< \lceil \frac{l-k}{k}\rceil}A_j$ is at most a rank-$(l-1)$  approximation of $A$. This implies that
    \begin{align*}
    \sigma_l(A)\leq \left\| A-\sum_{j< \lceil \frac{l-k}{k}\rceil}A_j\right\|_2&=\left\| \sum_{j\geq \lceil \frac{l-k}{k}\rceil}A_j \right\|_2\leq \sum_{j\geq \lceil \frac{l-k}{k}\rceil}\gamma e^{-\alpha j}=\\
    &=\gamma e^{-\alpha\lceil \frac{l-k}{k}\rceil}\sum_{j\geq 0}e^{-\alpha j}=\frac{\gamma}{1-e^{-\alpha}}\cdot e^{-\alpha \lceil \frac{l-k}{k}\rceil}.
    \end{align*}
    \end{proof} 
    \begin{lemma} \label{lem:sumdecay}
       Let $A =  \sum_{i = 1}^k A_i \in \mathbb{C}^{n \times n}$ 
       where
       $
         \sigma_j(A_i) \leq \gamma e^{- \alpha j}$, for $j = 1, \ldots, n$.
       Then 
       $
         \sigma_j(A) \leq \tilde \gamma e^{-\alpha \frac{j - k}{k}  }, \quad 
         \tilde \gamma = \frac{k\gamma}{1 - e^{-\alpha}}
       $.
     \end{lemma}
     
     \begin{proof}
       Relying on the SVD, we write $A_i=\sum_{j=1}^{\infty}\sigma_j(A_i)u_{i,j}v_{i,j}^*$ where $u_{i,j}$ and $v_{i,j}$ are the singular vectors of $A_i$ and where, for convenience, we have expanded the sum to an infinite number of terms by setting $\sigma_j(A_i)=0$ for $j>n$.
       This allows us to write
       \[
         A =  \sum_{i = 1}^k A_i = 
         \sum_{j = 1}^{\infty} \left(
            \sum_{i = 1}^k \sigma_j(A_i) u_{i,j} v_{i,j}^*
         \right) = \sum_{j = 1}^\infty \tilde A_j. 
       \]
       Observe that $\tilde A_j$ have rank
       $k$ and $\lVert A_j \rVert \leq k\gamma e^{-\alpha j}$. Applying Lemma~\ref{dyads} completes the proof. 
     \end{proof}
    \begin{lemma}\label{lem:rank-corr}
    Let $A,B\in\mathbb C^{m\times m}$  and suppose that $B$ has rank $k$. Then
    \[
    \sigma_{j+k}(A+B)\leq \sigma_j(A).
    \]
    \end{lemma}

    \begin{proof}
      For the Eckart-Young-Mirsky theorem $\forall j=1,\dots,m$
      $\exists \widetilde A$ of rank $j$ such that
      $\norm{A-\widetilde A}_2=\sigma_j(A)$. Therefore, since
      $\widetilde A+ B$ has rank less than or equal to $j+k$ we have
    \[
     \sigma_{j+k}(A+B)\leq \norm{(A+B)-(\widetilde A+B)}_2=\sigma_j(A).
    \]
    \end{proof}
   We are ready to study singular values of the matrix resulting from applying a function to a matrix. We prefer to begin by
   stating a simpler result which holds for matrices with spectrum contained in $B(0,1)$ and function holomorphic on a larger disk. In the following corollaries it is shown how to adapt this result to more general settings.
   \begin{theorem}\label{thm:matrix-func-decay}
   Let $A\in\mathbb C^{m\times m}$ be quasiseparable of rank $k$ and such that $A$ and all its trailing submatrices are enclosed in $(\rho,R_A,\mathcal V_A)$ and diagonalizable. Consider $f(z)$ holomorphic on $B(0,R)$ with $R>1$. Then, we can bound the singular values of a generic off-diagonal block $\tilde C$ in $f(A)$ with
   \[
   \sigma_l(\tilde C)\leq \gamma e^{-\frac{(\alpha+\alpha')l}{k}},\quad \alpha=\log\left(\frac{R_A}{\rho}\right),\quad \alpha'=\log( R),
   \]
   where $\gamma:= \max\limits_{|z|=R} |f(z)|\cdot  \kappa_{max}^2\cdot \norm {A}_2\cdot \Lambda(\rho,R_A,\mathcal V_A,R)\cdot \frac{k\cdot  \rho}{R R_A-\rho}$ and  $\kappa_{max}$ is the maximum among the spectral condition numbers of the trailing submatrices of $A$.
   \end{theorem}
   \begin{proof}
     Consider the partitioning
     $A= \left[\begin{smallmatrix} \bar A&\bar B\\ \bar C&\bar
         D \end{smallmatrix}\right]$ and for simplicity the case
     $k=1$, $\bar C=uv^t$. The general case is obtained by linearity
     summing $k$ objects of this kind coming from the SVD of $\bar C$
     and applying Lemma~\ref{lem:sumdecay}.  We rewrite the
     Dunford-Cauchy formula for $f(A)$
    \[
    f(A)=\frac{1}{2\pi i}\int_{S^1}(zI-A)^{-1}f(z)dz.
    \]
    Let $f(z)=\sum_{n\geq 0}a_nz^n$ be the Taylor expansion of $f(z)$ in $B(0,R)$. The corresponding off-diagonal block $\tilde C$ in $f(A)$ can be written as the outer product in Remark~\ref{rem:outer}
    \begin{equation}\label{representation}
    \left[
    u\ \vline\ \bar D\cdot u\ \vline\ \dots\ \vline\ \bar D^{s-1}\cdot u
    \right]\cdot
    \left[
    \sum_{n=0}^{s-1}a_{n+1}( A^t)^n\bar v\ \vline\ \dots\ \vline\ (a_s  A^t+a_{s-1}I)\bar v\ \vline\  a_s \bar v
    \right]^t[ I \ 0 ]^t+g_s(A),
    \end{equation}
     where $\bar v=[I\  0]^tv$ and $g_s(A)$ is the remainder of the truncated Taylor series at order $s$. Since $f(z)$ is holomorphic in $B(0,R)$ the coefficients of $f(z)$ verify \cite[Theorem 4.4c]{henrici}
     \[
     |a_j|\leq \max_{|z|=R}|f(z)| \cdot R^{-j}.
     \]
     Applying Theorem~\ref{thm:kryl-horn-decay} we get that $\forall r\in(\rho,R_{A})$  
      \[
      \sigma_l(\tilde C-g_s(A))\leq\gamma e^{-(\alpha+\alpha') l},
      \]
     with $\alpha,\alpha',\delta,\kappa_{max}$ as in the thesis and $\gamma= \max\limits_{|z|=R} |f(z)|\cdot  \kappa_{max}^2\norm {A}_2\cdot \Lambda(\rho,R_A,\mathcal V_A,R)$.
     Observing that this bound is independent on $s$ and $\lim_{s\to\infty}g_s(A)=0$ we get the thesis.
   \end{proof}
   \begin{corollary}\label{cor:decay-func}
   Let $A\in\mathbb C^{m\times m}$ be a $k$-quasiseparable matrix, $z_0\in\mathbb C$ and $R'\in\mathbb R^+$ such that $R'^{-1}(A-z_0I)$ is enclosed in $(\rho,R_A,\mathcal V_A)$. Then, for any holomorphic function $f(z)$ in $B(z_0,R)$ with $R>R'$, any off-diagonal block $\tilde C$ in $f(A)$ has singular values bounded by
   \[
   \sigma_l(\tilde C)\leq  \gamma e^{-\frac{(\alpha+\alpha')l}{k}},\quad \alpha=\log\left(\frac{R_A}{\rho}\right),\quad \alpha'=\log\left( \frac{R}{R'}\right),
   \] 
   where $\gamma:= \max\limits_{|z-z_0|=R} |f(z)|\cdot  \kappa_{max}^2\cdot\norm {A-z_0I}_2\cdot \Lambda(\rho,R_A,\mathcal V_A,R)\cdot \frac{k\cdot  \rho}{R R_A-\rho R'}$ and $\kappa_{max}$ is the maximum among the spectral condition numbers of the trailing submatrices of $R'^{-1}(A-z_0I)$.
   \end{corollary}
   \begin{proof}
   Define $g(z)=f(R'z+z_0)$ which is holomorphic on $B(0,\frac{R}{R'})$. Observing that $f(A)=g(R'^{-1}(A-z_0I))$ we can conclude by applying Theorem~\ref{thm:matrix-func-decay}.
   \end{proof}
    \begin{remark}
    If we can find $z_0\in\mathbb C$ such that $\norm{A-z_0I}_2<R$ then it is always possible to find $(\rho,R_A,\mathcal V_A)$ with $\rho\in(0,1)$ which satisfies the hypothesis of the previous corollary. A worst case estimate for $\frac{\rho}{R_A}$ is $\frac{\norm{A-z_0I}_2}{R}$ since this  is the radius of a circle containing the spectrum of the rescaled matrix and --- given that the Riemann map for a ball centered in $0$ is the identity --- $R_A=1$.
    \end{remark}
    \begin{example}[Real spectrum]
    We here want to estimate the quantity $\frac{R_A}{\rho}$ in the case of a real spectrum for the matrix $A$. Suppose that --- possibly after a scaling --- the latter is contained in the symmetric interval $[-a,a]$ with $a\in(0,1)$. The logarithmic capacity of this set is $\frac{a}{2}$ and the inverse of the associated Riemann map is $\psi(z)=z+\frac{a^2}{4}$. This follows by observing that the function $z+z^{-1}$ maps the circle of radius $1$ into $[-2,2]$, so then it is sufficient to compose the latter with two homothetic transformations to get $\psi(z)$. Moreover, observe that --- given $r\geq\frac{a}{2}$ --- $\psi$ maps the circle of radius $r$ into an ellipse of foci $[-a,a]$. Therefore, in order to get $R_A$ it is sufficient to compute for which $r$ we have $\psi(r)=1$. This corresponds to finding the solution of $r+\frac{a^2}{4r}=1$ which is greater than $\frac{a}{2}$. This yields
    \[
    R_A=\frac{1+\sqrt{1-a^2}}{2}\quad \Rightarrow\quad \frac{R_A}{\rho}=\frac{1+\sqrt{1-a^2}}{a}.
    \]
    
    \end{example}

   \section{Functions with singularities}
   \label{sec:poles}
   If some singularities of $f$ lie inside $B(z_0,R)$ then $f(A)\neq \int_{\partial B(z_0,R)}f(z)(zI-A)^{-1}dz$. However, since the coefficients of the Laurent expansion of 
   $f$ with negative degrees in \eqref{cbar} do not affect the result, the statement of Theorem~\ref{thm:matrix-func-decay} holds for the matrix $\int_{\partial B(z_0,R)}f(z)(zI-A)^{-1}dz$. In this section we prove that --- under mild conditions --- the difference of the above two terms still has a quasiseparable structure. This  numerically preserves the  quasiseparability of $f(A)$.
   \subsection{An extension of the Dunford-Cauchy integral formula}
   
   The main tool used to overcome difficulties in case
   of removable singularities will be the following
   result, which is an extension of the integral 
   formula used in Definition~\ref{def:matrix-func1}. 
   
   \begin{theorem}\label{thm:ext-cauchy}
   Let $f(z)$ be a meromorphic function with a discrete set of poles $\mathcal P$ and $A\in\mathbb C^{m\times m}$ with spectrum $\mathcal S$ such that $\mathcal S\cap\mathcal P=\emptyset$. Moreover, consider $\Gamma$ simple closed curve in the complex plane which encloses $\mathcal S$ and $T:=\{z_1,\dots,z_t\}\subseteq \mathcal P$ subset of poles with orders $d_{1},\dots,d_{t}$ respectively. Then
   \[
   \frac{1}{2\pi i}\int_{\Gamma}(zI-A)^{-1}f(z)dz=f(A)+\sum_{j=1}^tR_j(z_jI-A),
   \]
   where $R_j$ is the rational function
   \[
   R_j(z):=\sum_{l=1}^{d_j}(-1)^{l+1}\frac{f_j^{(d_j-l)}(z_j)}{(d_j-l)!}z^{-l}
   \]
   and  $f_j(z)=(z-z_j)^{d_j}f(z)$, extended to the limit in $z_j$.    
   In particular if the poles in $T$ are simple then
   \[
                   \frac{1}{2\pi i}\int_{\Gamma}(zI-A)^{-1}f(z)dz=f(A)+\sum_{j=1}^tf_j(z_j)\cdot(z_jI-A)^{-1}=f(A)+\sum_{j=1}^tf_j(z_j)\mathfrak R(z_j).
   \]                
   \end{theorem}
   \begin{proof}
   We first prove the statement for $A$ diagonalizable. Assume that
   $V^{-1}AV=\hbox{diag}(\lambda_1,\ldots,\lambda_n)$, then
   \begin{equation}\label{residue}
   \frac{1}{2\pi i}\int_{\Gamma}(zI-A)^{-1}f(z)dz=V^{-1}\begin{bmatrix} \frac{1}{2\pi i}\int_{\Gamma}\frac{f(z)}{z-\lambda_1}\\
                    &\ddots\\
                    &&\frac{1}{2\pi i}\int_{\Gamma}\frac{f(z)}{z-\lambda_m}
                    \end{bmatrix}V.
   \end{equation}
   Applying the Residue theorem we arrive at
   \[
   \frac{1}{2\pi i} \int_{\Gamma}\frac{f(z)}{z-\lambda_p}=\res\left(\frac{f}{z-\lambda_p},\lambda_p\right)+\sum_{j=1}^t\res\left(\frac{f}{z-\lambda_p},z_j\right),\qquad p=1,\dots,m.
   \]
   Since $\lambda_p$ is a simple pole of $\frac{f}{z-\lambda_p}$ the first summand is equal to $f(\lambda_p)$. 
   
   On the other hand $z_j$ is a pole of order $d_j$ of $\frac{f}{z-\lambda_p}$, therefore its residue is
   \[
   \res\left(\frac{f}{z-\lambda_p},z_j\right)=
   \frac{1}{(d_j-1)!}\lim_{z\to z_j}\frac{\partial^{d_j-1}}{\partial z^{d_j-1}}\left((z-z_j)^{d_j}\frac{f}{z-\lambda_p}\right)=\frac{1}{(d_j-1)!}\frac{\partial^{d_j-1}}{\partial z^{d_j-1}}\left(\frac{f_j}{z-\lambda_p}\right)(z_j).
   \]
   One can prove by induction (see Appendix) that, given a sufficiently differentiable $f_j(z)$, it holds
   \begin{equation}\label{ind1}
   \frac{\partial^{d-1}}{\partial z^{d-1}}\left(\frac{f_j(z)}{z-\lambda_p} \right)=
   \sum_{l=1}^{d}(-1)^{l+1}\frac{(d-1)!}{(d-l)!}f_j^{(d-l)}(z)(z-\lambda_p)^{-l},\quad d\in\mathbb N.
   \end{equation}
   Setting $d=d_j$ in \eqref{ind1} we  derive
   \[
   \res\left(\frac{f}{z-\lambda_p},z_j\right)=R_j(z_j-\lambda_p).
   \]
   To conclude it is sufficient to  rewrite the diagonal matrix in \eqref{residue} as
   \[
   \begin{bmatrix} f(\lambda_1)\\
                 &\ddots\\
                 &&f(\lambda_m)
                 \end{bmatrix}+\sum\limits_{j=1}^t
   \begin{bmatrix}  R_j(z_j-\lambda_1)\\
                &\ddots\\
                && R_j(z_j-\lambda_m)
                \end{bmatrix}.             
   \]
   We now prove the thesis for
   \[
   A=\begin{bmatrix}
   \lambda&1\\
   &\ddots&\ddots\\
   &&\ddots&1\\
   &&&\lambda
   \end{bmatrix},
   \]
   because the general non diagonalizable case can be decomposed in sub-problems of that kind. We have that
   \[                 
   \frac{1}{2\pi i}\int_{\Gamma}(zI-A)^{-1}f(z)dz=\frac{1}{2\pi i}\begin{bmatrix} 
   \int_{\Gamma}\frac{f(z)}{z-\lambda}&\int_{\Gamma}\frac{f(z)}{(z-\lambda)^2}&\dots&\int_{\Gamma}\frac{f(z)}{(z-\lambda)^m}\\
             &\ddots&\ddots&\vdots\\
             &&\ddots&\int_{\Gamma}\frac{f(z)}{(z-\lambda)^2}\\
             &&&\int_{\Gamma}\frac{f(z)}{z-\lambda}
             \end{bmatrix}. 
   \]
   In order to reapply the previous argument  is sufficient to prove that
   \begin{itemize}
   \item[(i)] $\res(\frac{f}{(z-\lambda)^{h+1}},\lambda)=\frac{f_j^{(h)}(\lambda)}{h!}$ $h=1,\dots,m-1$,
   \item[(ii)] $\res(\frac{f}{(z-\lambda)^{h+1}},z_j)=\frac{R_j^{(h)}(z_j-\lambda)}{h!}$ $h=1,\dots,m-1$.
   \end{itemize}
   The point $(i)$ is a direct consequence of the fact that $\lambda$ is a pole of order $h+1$ of the function $\frac{f(z)}{(z-\lambda)^{h+1}}$. 
   Concerning $(ii)$ observe that $z_j$ is again a pole of order $d_j$ for the function $\frac{f(z)}{(z-\lambda)^{h+1}}$ so \[
      \res \left(  
         \frac{f}{(z-\lambda)^{h+1}},z_j\right) =
         \frac{1}{(d_j-1)!}   \frac{\partial^{d_j-1}}{\partial z^{d_j-1}}\left(\frac{f_j(z)}{(z-\lambda)^{h+1}}\right)(z_j).\] 
   One can prove by induction (see Appendix) that, 
   for each $d, h \in \mathbb N$: 
   \begin{equation}\label{ind2}
   \frac{\partial^{d-1}}{\partial z^{d-1}}\left(\frac{f_j(z)}{(z-\lambda)^{h+1}} \right)=\frac{(d-1)!}{h!}\sum_{l=1}^{d}(-1)^{l+h+1}\frac{(l+h-1)!}{(d-l)!(l-1)!} f_j^{(d-l)}(z)(z-\lambda)^{-(h+l)}
   \end{equation}
   Successive derivation of $R_j$ 
   repeated $h$ times yields:
   \[
   R_j^{(h)}(z)=\sum_{l=1}^{d_j}(-1)^{l+h+1}\frac{(l+h-1)!}{(d_j-l)!(l-1)!} f_j^{(d_j-l)}(z_j)z^{-(h+l)}, 
   \]
   and by setting $d=d_j$ in \eqref{ind2} we finally 
   get $(ii)$.
   \end{proof}
   \subsection{Functions  with poles}
   As a direct application of Corollary~\ref{cor:decay-func} we can give a concise statement in the case of simple poles.
   \begin{corollary}\label{cor:poles}
   Let $A\in\mathbb C^{m\times m}$ be a quasiseparable matrix with rank $k$, $z_0\in\mathbb C$ and $R'\in\mathbb R^+$ such that $R'^{-1}(A-z_0I)$ is enclosed in $(\rho,R_A,\mathcal V_A)$.
   Consider $R>R'$ and a function $f(z)$ holomorphic on the annulus $\mathcal A:=\{R'<|z-z_0|<R\}$. If the disc $B(z_0,R')$ contains $t$ simple poles of $f$ then  any off-diagonal block $\tilde C$ in $f(A)$ has singular values bounded by
   \[
   \sigma_l(\tilde C)\leq  \gamma e^{-\frac{(\alpha+\alpha')(l-tk)}{k}},\quad \alpha=\log\left(\frac{R_A}{\rho}\right),\quad \alpha'=\log\left( \frac{R}{R'}\right),
   \] 
   where $\gamma:= \max\limits_{|z-z_0|=R} |f(z)|\cdot  \kappa_{max}^2\cdot\norm {A-z_0I}_2\cdot \Lambda(\rho,R_A,\mathcal V_A,R)\cdot \frac{k\cdot  \rho}{R R_A-\rho R'}$ and $\kappa_{max}$ is the maximum among the spectral condition numbers of the trailing submatrices of $R'^{-1}(A-z_0I)$.
   \end{corollary}
   \begin{proof}
   Let $f(z)=\sum\limits_{n\in\mathbb Z}a_nz^n$ be the series expansion of $f$ in $\mathcal A$ and $z_{1},\dots,z_{t}$ be the simple poles of $f$ inside $B(z_0,R')$. Then
   \[
   |a_j|\leq \norm{f(z)}_{\infty, \partial B(z_0,R)}\cdot \left(\frac{R'}{R}\right)^j,\qquad n\geq 0.
   \]
   According to what we observed at the beginning of Section~\ref{sec:poles} we can apply Corollary~\ref{cor:decay-func} to the off-diagonal singular values of $B:=\int_{\partial B(z_0,R')}f(z)(zI-A)^{-1}dz$. Moreover, using Theorem~\ref{thm:ext-cauchy} we get
   \[
   f(A)=B-\sum_{j=1}^tf_{j}(z_{j})\cdot(z_{j}I-A)^{-1}.
   \] 
   Observing that the right summand has at most quasiseparable rank $tk$ we can conclude, using Lemma~\ref{lem:rank-corr}, that the bound on the singular values of $f(A)$ is the same which holds for $B$, but shifted by the quantity $t\cdot k$.
   \end{proof}

   \subsection{Functions with essential singularities}
   Consider the case of a function $f(z)$ holomorphic in $\mathbb C\setminus\{a\}$ with an essential singularity in $a$. Moreover, suppose that $a$  is not an eigenvalue of the argument $A\in\mathbb C^{m\times m}$. In a suited punctured disk $B(a,R)\setminus\{a\}$ --- which contains the spectrum of $A$ --- we can expand $f$ as
   \[
   f(z):=\sum_{n\in\mathbb Z}a_n(z-a)^n.
   \] 
   In particular we can decompose $f$ as $f_1(z-a)+f_2((z-a)^{-1})$ with $f_i$ holomorphic on $B(0,R)$ for $i=1,2$. 
   Therefore
   \[
   f(A)=f_1(A-aI)+f_2((A-aI)^{-1}).
   \]
   Since $f_1$ and $f_2$ are both holomorphic and the operations of shift and inversion preserve the quasiseparable rank we can apply Theorem~\ref{thm:matrix-func-decay} and Lemma~\ref{lem:sumdecay} in order to get estimates on the off-diagonal singular values of $f(A)$.
   
   One can use this approach in the case of finite order poles and find equivalent bounds to Corollary~\ref{cor:poles}, although in a less explicit form.   
   
   \subsection{Functions with branches}
   
   We conclude this section describing how to re-adapt the approach in the case of functions with multiple branches. The same trick can be used to deal with other scenarios, such as the presence
   of singularities that has been described previously. 
   
   The main idea is that, in the integral definition of a matrix function, the 
   path $\Gamma$ does not need to be a single Jordan curve, but can be 
   defined as a union of a finite number of them. The only requirement is that the 
   function is analytic in the Jordan regions, and that the spectrum
   is contained in their union. 
   
   In our setting, it might happen that we cannot enclose the spectrum in a single
   ball without capturing also the branching point. However, it is always possible to
   cover it with the union of a finite number of such balls. In this context, 
   assuming that the path $\Gamma$ is split as the borders of $t$ balls, denoted
   by $\Gamma_1, \ldots, \Gamma_t$, one has 
   \[
     f(A) = \sum_{i = 1}^t \int_{\Gamma_i} f(z) \mathfrak R(z) dz. 
   \]
   Assuming that the number $t$ is small enough, we can obtain the 
   numerical Quasiseparability of $f(A)$ by the quasiseparability
   of each of the addends and then
   relying on Lemma~\ref{lem:sumdecay}. 
   Inside each $\Gamma_i = B(z_i, r_i)$ 
   we can perform the change of variable
   $\tilde z := r_i (z - z_i)$ 
   and write the resolvent as (here the coefficient
   $D$ will be different by scaling and translation in every $\Gamma_i$): 
   \[
   \mathfrak R(\tilde z) = \begin{bmatrix}
   * & * \\
   (\tilde zI - D)^{-1} C(\tilde z) S_{D}(\tilde z)^{-1} & * \\
   \end{bmatrix}, 
   \qquad 
   \begin{cases}
     (\tilde zI - D)^{-1} = \sum_{j \in \mathbb Z} D_j \tilde z^j \\
     S_D^{-1}(\tilde z) = \sum_{s \in \mathbb Z} H_s \tilde z^s \\
   \end{cases}
   \]
   The construction of the coefficients $D_j$ can be done
   by writing $D$ in Jordan canonical from as 
   \[
     V^{-1} D V = \begin{bmatrix}
       J_{\text{in}} \\
       & J_{\text{out}} \\
     \end{bmatrix}, \qquad 
     V = \begin{bmatrix}
       V_{1} & V_{2} \\
     \end{bmatrix}, \quad 
     V^{-1} = \begin{bmatrix}
       W_{1}  \\
       W_{2}  \\
     \end{bmatrix}
   \]
   where $J_{\text{in}}$ refers to the part of the spectrum 
   inside $\Gamma_i$, and $J_{\text{out}}$ the one outside. 
   Thanks to the change of variable in the integral, this
   corresponds to asking that
   the spectrum of $J_{\text{in}}$ is inside the unit disc, and
   the one of $J_{\text{out}}$ outside. 
   Then, one has the following definition for $D_j$: 
   \[
     D_j = \begin{cases}
        V_{1} J_{\text{in}}^{-j-1} W_{1} & j < 0 \\
       -V_{2} J_{\text{out}}^{-j-1} W_{2} & j \geq 0 \\
     \end{cases},
   \]
   and an analogous formula holds for the coefficients $H_s$. This
   provides the Laurent expansion of the off-diagonal block
   in the integrand. A similar analysis to the one carried
   out in the previous sections can be used to retrieve
   the decay on the singular values of this block.

       \section{Computational aspects and validation of the bounds}
       \label{sec:computational}
       
       In the previous sections
       we have proved that the numerical
       quasiseparable structure is often present
       in $f(A)$. This property can be used to speed up the matrix arithmetic operations and then to efficiently evaluate $f(A)$ by means of contour integration. We briefly describe the strategy in the next subsections and we refer the reader to \cite{hackbusch2016hierarchical} for more details. In Section~\ref{sec:validation} we will compare our bounds with the actual decay in some concrete cases.
       \subsection{Representation and arithmetic operations}
        In order to take advantage of the quasiseparable structure we need a representation that enable us to perform the storage and the matrix operations cheaply. 
        We rely on the framework of Hierarchical representations originally introduced by Hackbusch \cite{hackbusch1999sparse,hackbusch2016hierarchical} in the context of integral and partial differential equations. It consists in a class of recursive block representations with structured sub-matrices that allows the treatment of a number of data-sparse patterns. Here, we consider a particular member  of this family --- sometimes called Hierarchical off-diagonal low-rank representation (HODLR) --- which has a simple formulation and an effective impact in handling quasiseparable matrices.  
        
        Let $A\in\mathbb{C}^{m\times m}$ be a $k$-quasiseparable matrix and consider the partitioning
            \[A=\left[\begin{smallmatrix}A_{11}&A_{22}\\
            A_{21}&A_{22}\end{smallmatrix}\right],\]where 
            $A_{11}\in\mathbb C^{m_1\times m_1}$, 
              $A_{22}\in\mathbb C^{m_2\times m_2}$,
              with $m_1:=\lfloor \frac{m}{2} \rfloor $ and $m_2:=\lceil \frac{m}{2} \rceil$. 
              Observe that the antidiagonal blocks $A_{12}$ and $A_{21}$ do not involve any element of the main diagonal of $A$, hence we
              can represent them
              in a compressed form as an outer product of rank $k$.
              Moreover, the diagonal blocks $A_{11}$ and $A_{22}$ are square matrices which are again $k$-quasiseparable. Therefore it is possible to re-apply this procedure recursively. We stop when the diagonal blocks reach a minimal dimension $m_{\text{min}}$, and we store them as full matrices. The process is described graphically in Figure~\ref{fig:Hmatrices}.
              
              \begin{figure}[!ht]
            \centering
            \includegraphics[width=0.2\textwidth]{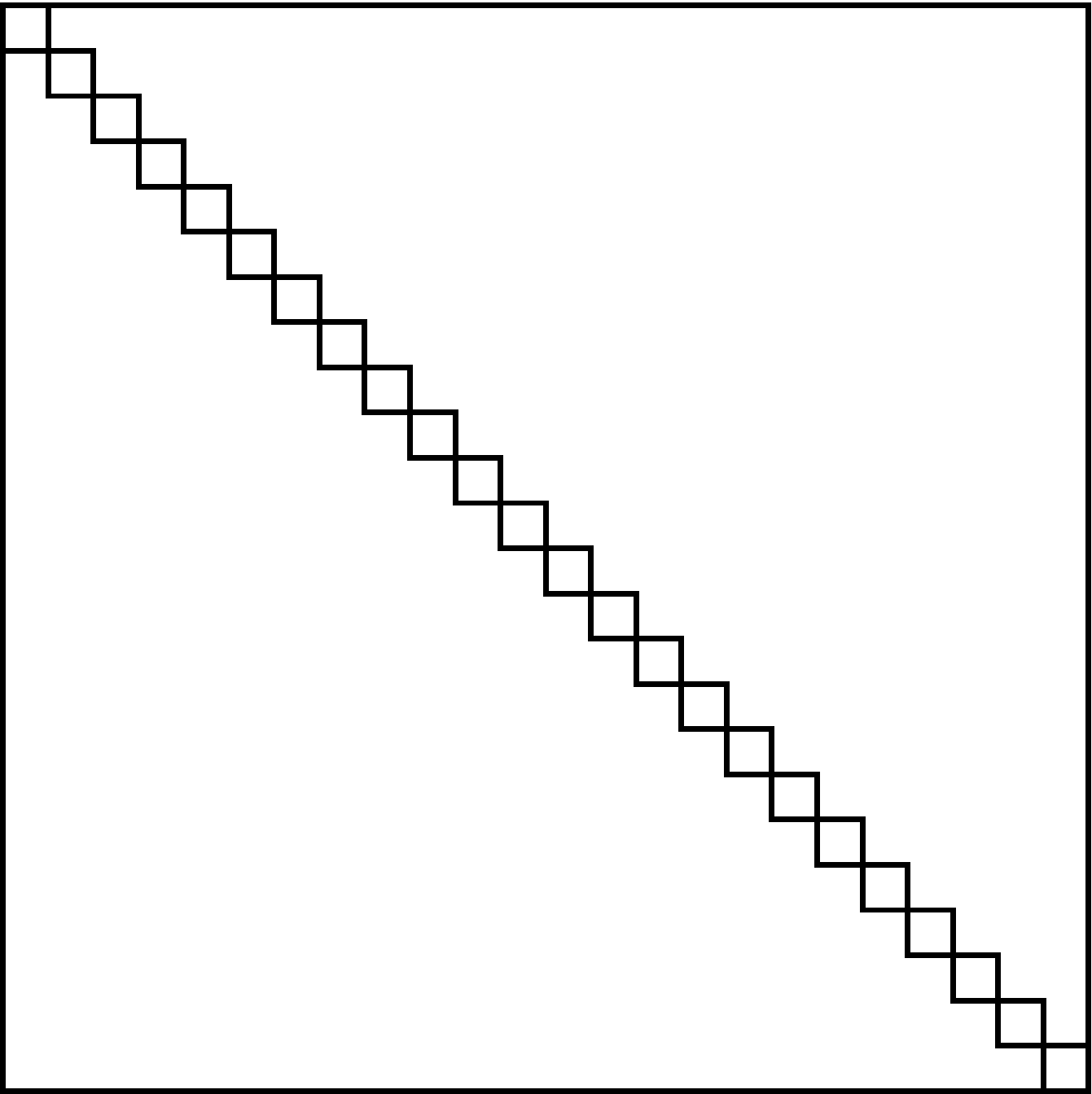}\qquad 
            \includegraphics[width=0.2\textwidth]{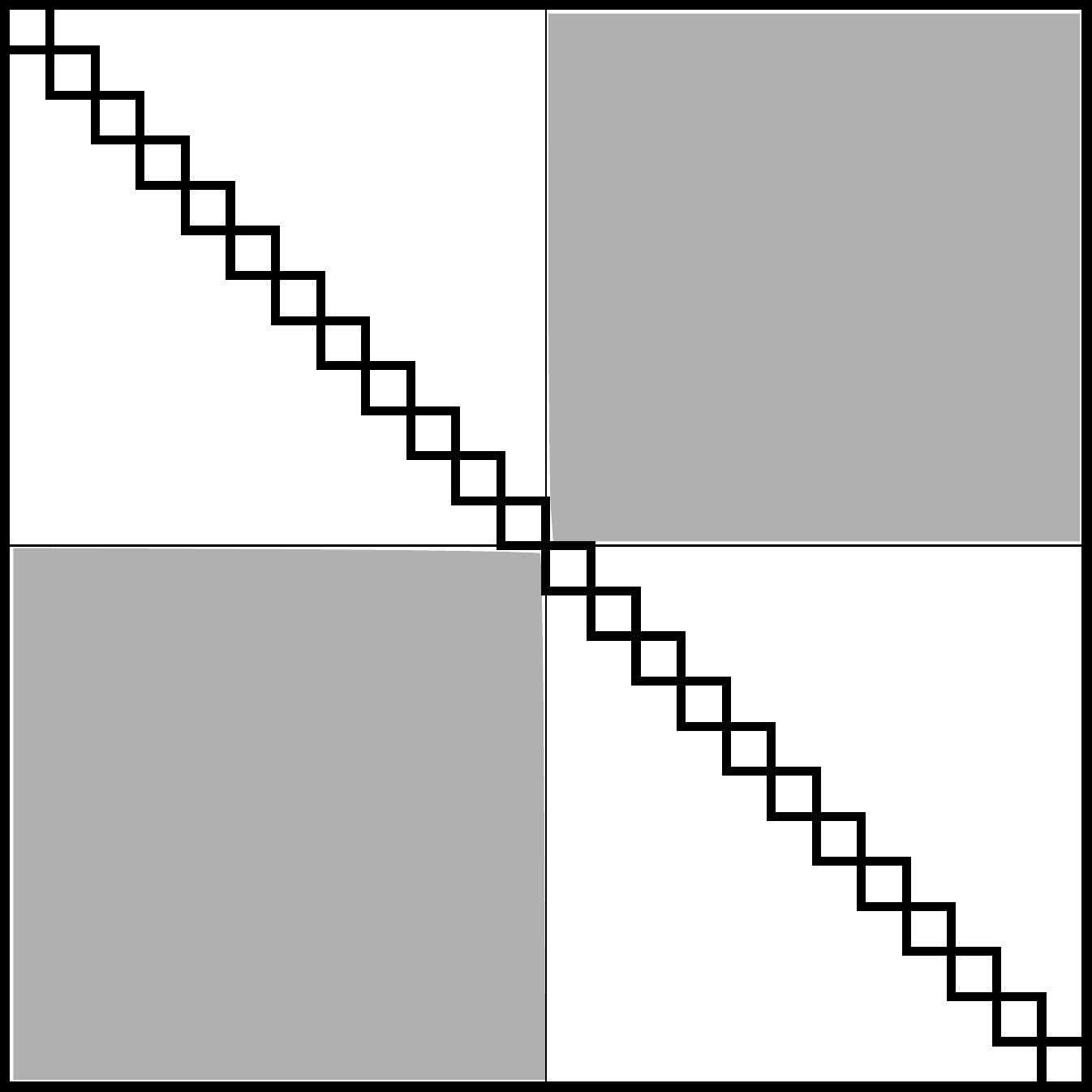}
            \qquad 
            \includegraphics[width=0.2\textwidth]{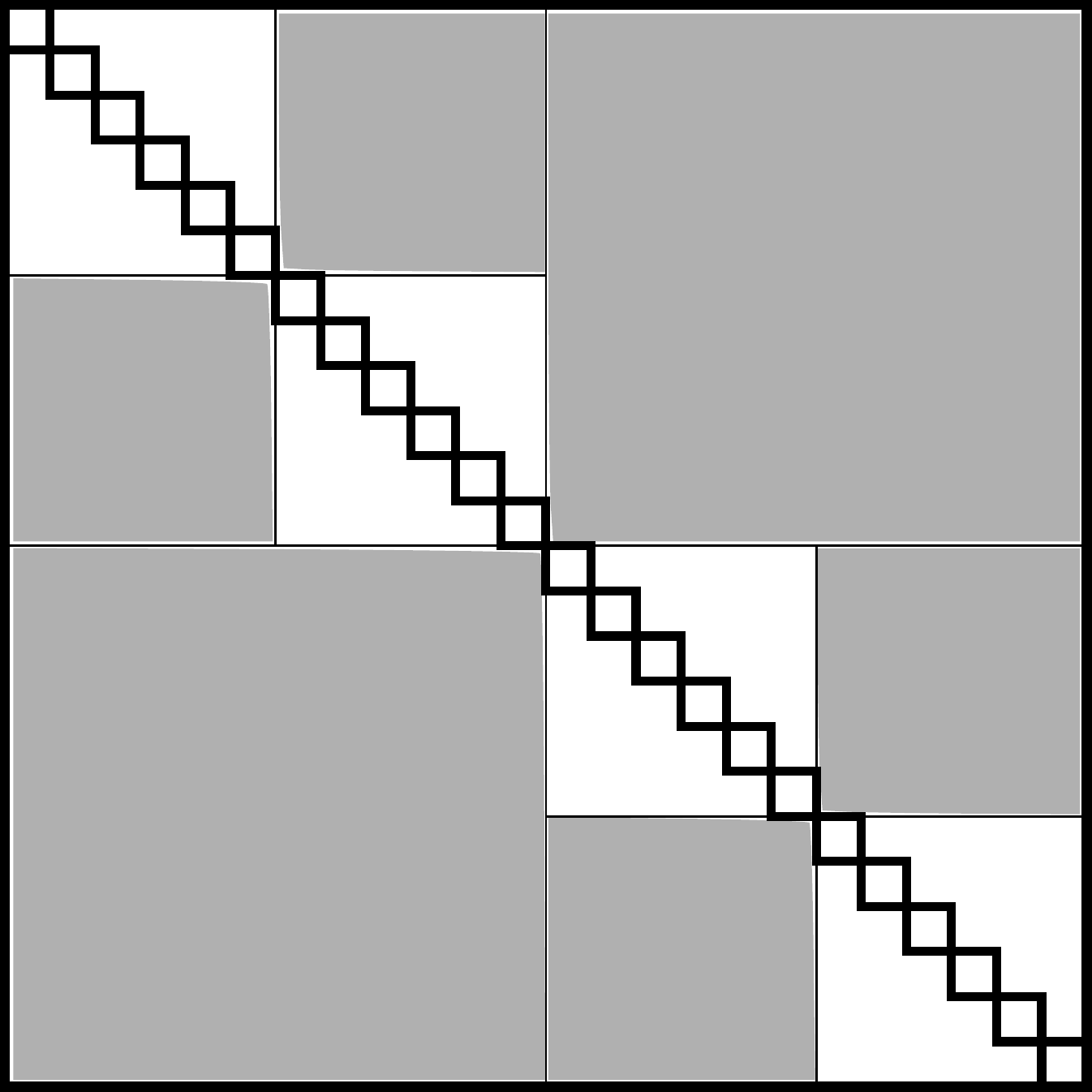}\qquad 
            \includegraphics[width=0.2\textwidth]{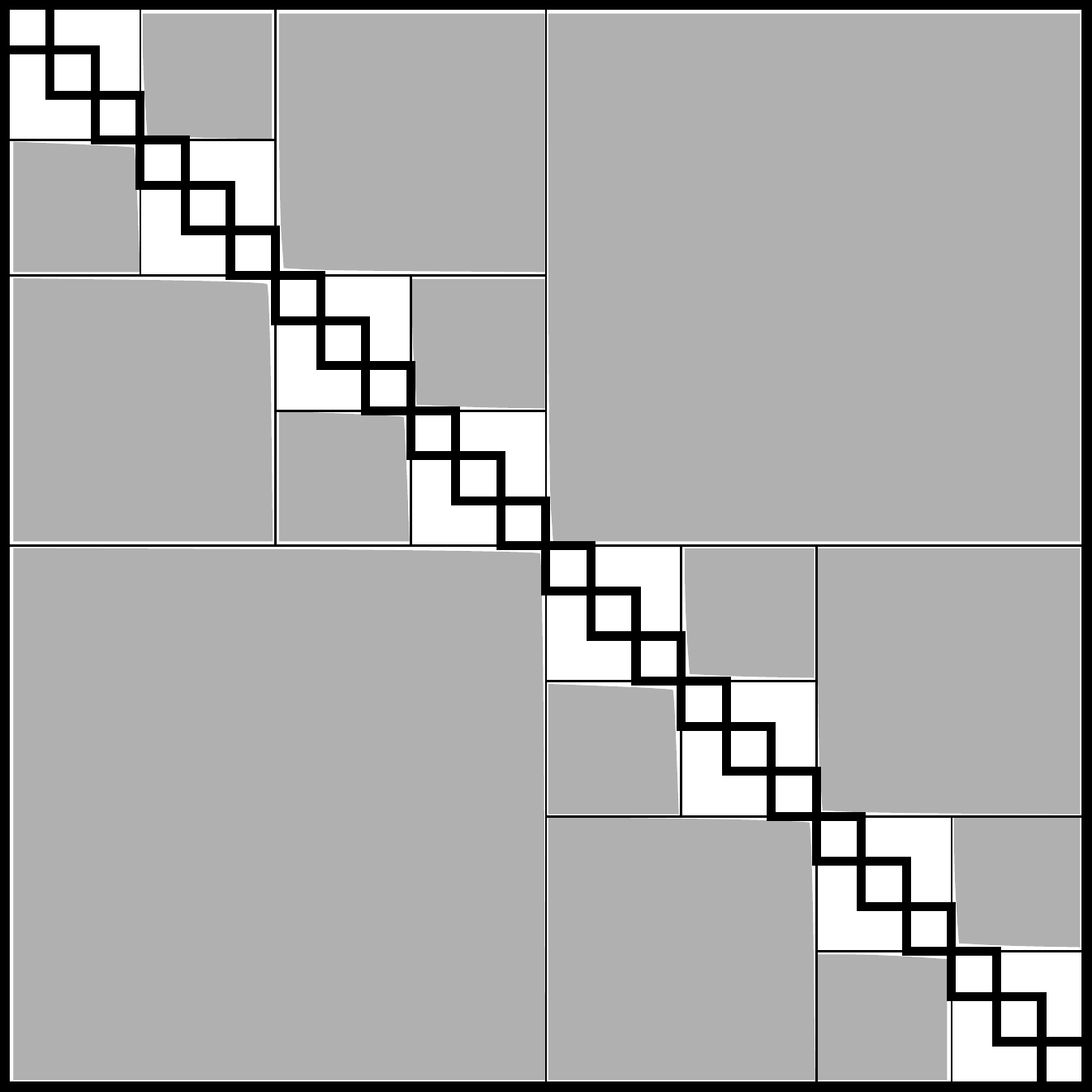}
             \caption{The behavior of the block partitioning in the HODLR-matrix representation. The blocks filled with grey are low rank matrices represented in a compressed form, and the diagonal blocks in the last step are stored as dense matrices.}\label{fig:Hmatrices}
            \end{figure}
            
            If $m_{\text{min}}$ and $k$ are negligible with respect to $m$ then the storage cost of each sub-matrix is $O(m)$. Since the levels of the recursion are $O(\log(m))$, this yields a linear-polylogarithmic memory consumption with respect to the size of the matrix.
            
            The HODLR representation acts on a matrix by compressing many of its sub-blocks. Therefore, it is  natural to perform the arithmetic operations in a block-recursive fashion. The basic steps of these procedures require arithmetic operations between low-rank matrices or  $m_{min}\times m_{min}$-matrices. If the rank of the off-diagonal blocks is small compared to $m$, then the algorithms performing the arithmetic operations have linear polylogarithmic complexities \cite{borm}[Chapter 6]. The latter are summarized in Table~\ref{tab:complexity} where it is assumed that the constant $k$ bounds the 
            quasiseparable rank of all the matrices involved. Moreover, the operations are performed adaptively with respect to the rank of the blocks. This means that the result
            of an arithmetic operation will be an HODLR matrix with the same partitioning, where each low rank block is a truncated reduced SVD of the corresponding block of
            the exact result. This operation can be carried out with linear cost, assuming the quasiseparable stays negligible with respect to $m$.   Hence the rank is not fixed a priori but depends on a threshold $\epsilon$ at which the
            truncation is done. We refer to
                        \cite{hackbusch2016hierarchical} for a complete description. In our experiments we set $\epsilon$ equal to the machine precision $2.22\cdot 10^{-16}$ and $m_{\text{min}}=64$.
            \begin{table}
                                \begin{center}                   
                                \resizebox{0.8\textwidth}{!} {
                 \footnotesize \begin{tabular}{cc}
                  \hline
                   Operation & Computational complexity \\ \hline  
                  Matrix-vector multiplication & $O(k m\log(m))$\\
                  Matrix-matrix addition & $O(k^2 m\log(m))$\\
                  Matrix-matrix multiplication & $O(k^2 m\log(m)^2)$\\
                  Matrix-inversion & $O(k^2 m\log(m)^2)$\\
                  Solve linear system & $O(k^2 m\log(m)^2)$\\
                   \hline
            
                 \end{tabular}
                }
                \end{center}
                \caption{Computational complexity of the HODLR-matrix arithmetic. The operation \emph{Solve linear system}
                 comprises to compute the LU factorization of the coefficient matrix and to solve the two triangular linear systems}\label{tab:complexity}
                \end{table}
       \subsection{Contour integration}
       
       The Cauchy integral formula \eqref{cauchyformula}
       can be used to approximate $f(A)$ by means
       of a numerical integration scheme. Recall that,
       given a complex valued function $g(x)$ defined
       on an interval $[a,b]$ one can approximate
       its integral by
       \begin{equation} \label{quadrature}
         \int_{a}^b g(x) dx\ \approx \ \sum_{k = 1}^N w_k
         \cdot g(x_k)
       \end{equation}
       where $w_k$ are the \emph{weights} and 
       $x_k$ are the \emph{nodes}. Since we
       are interested in integrating 
       a function on $S^1$ 
       we can write
       \[
         \frac{1}{2\pi i}\int_{S^1} f(z) (zI - A)^{-1} dz = 
         \frac{1}{2\pi}\int_{0}^{2\pi} f(e^{ix}) (I - e^{-ix} A)^{-1}  dx, 
       \]
       where we have parametrized $S^1$ by means of 
       $e^{ix}$. The right-hand side can be approximated
       by means of \eqref{quadrature}, 
       so we obtain:
       \begin{equation}\label{quadrature2}
        f(A) \approx \frac{1}{2\pi}\sum_{k = 1}^N w_k
              \cdot f(e^{ix_k}) (I - e^{-ix_k} A)^{-1} = 
           \frac{1}{2\pi}\sum_{k = 1}^N w_k
                      \cdot e^{ix_k} f(e^{ix_k})\mathfrak R(e^{ix_k}).   
       \end{equation}
       
       \begin{algorithm}
       \caption{Pseudocode for the evaluation of a contour integral on $S^1$}\label{alg:contour}
       \begin{algorithmic}[1]
       \Procedure{ContourIntegral}{$f,A$}\Comment{Evaluate $\frac{1}{2\pi i}\int_{S^1}f(z)(zI-A)^{-1}dz$}
       \State $N\gets1$
       \State $M\gets f(1)\cdot(I-A)^{-1}$
       \State $err\gets\infty$
       \While{$err> \sqrt u$}
       \State $M_{\text{old}}\gets M$
       \State $M\gets\frac 12 M$\Comment{The new weights are applied to the old evaluations}
       \State $N\gets 2N$
       \For{$j=1,3,\dots,N-1$}\Comment{Sum the evaluations on the new nodes}
       \State $z \gets e^{\frac{2 \pi i j}{N}}$
       \State $M\gets M + \frac{z f(z)}N \cdot (zI - A)^{-1}$
       \EndFor
       \State $err \gets \lVert M - M_{\text{old}} \rVert_2$
       \EndWhile
       \State \Return $M$
       \EndProcedure
       \end{algorithmic}
       \end{algorithm}
       This approach has already been explored in \cite{gavrilyuk2002mathcal}, mainly for the computation
           of $f(A) b$ due to the otherwise high cost
           of the inversions in the general case. The pseudocode of the procedure is reported in Algorithm~\ref{alg:contour}.
           
            Algorithm~\ref{alg:contour} --- based on \eqref{quadrature2} --- can be carried out cheaply when
           $A$ is represented as an HODLR-matrix, since the inversion
           only requires $O(m \log^2(m))$ flops. Moreover, not only
           the resolvent $\mathfrak R(e^{ix_k})$ is representable as
           a HODLR-matrix, but the same holds for the final result
           $f(A)$ in view of Theorem~\ref{thm:matrix-func-decay}. This
           guarantees the applicability of the above strategy even
           when dealing with large dimensions.
       
        The results in Section~\ref{sec:poles} 
       enable us to deal with functions having
       poles inside the domain of integration. The only
       additional step that is required is to 
       compute the correction term
       described in Theorem~\ref{thm:ext-cauchy}. Notice
       that this step just requires additional 
       evaluations
       of the resolvent
       and so does not change the asymptotic complexity of the
       whole procedure. 
       
       We show now an example where Theorem~\ref{thm:ext-cauchy} can be used to derive
       an alternative algorithm for the evaluation of
       matrix functions with poles inside the domain. 
       
       More precisely, we consider a matrix $A$ with 
       spectrum contained in the unit disc, and the evaluation
       of the matrix function $f(A)$ with $f(z) = \frac{e^{z}}{\sin(z)}$. Applying Theorem~\ref{thm:ext-cauchy} yields
       \[
         f(A) = \int_{S^1} f(z) \mathfrak R(z) dz
           + A^{-1}. 
       \]
       One can then choose to obtain $f(A)$ by computing
       $e^A \cdot ( \sin A )^{-1}$, which requires the 
       evaluation of two integrals and one inverse, or 
       using the above formula, which only requires 
       one integral and an inverse. 
       
        We used an adaptive doubling strategy for the number of nodes i.e., starting with $N$-th roots of the unit for a small value of $N$. We apply the quadrature rule \eqref{quadrature2}  and  we double $N$ until the quality of the approximation is satisfying. In order to check this, we require that the norm of the difference between two consecutive  approximations is smaller than a certain threshold. The $2$-norm of an HODLR-matrix can be estimated in linear time as shown in \cite{hackbusch2016hierarchical}.  Since the quadrature rule is quadratically convergent \cite{trefethen2014exponentially} and the  magnitude of the distance between the approximations at step $k$ and $k+1$ is a heuristic estimate for the error at step $k$ we choose as threshold $\sqrt u$ where $u$ is the unit round-off. In this way we should get an error of the order of $u$.
        
       We show in Table~\ref{tab:sin}, where the approach relying
       on Theorem~\ref{thm:ext-cauchy} and 
       on computing the function separately are
        identified by
       the labels
       ``sum'' and ``inv'', respectively, 
       that the first choice is
       faster (due to the reduced number of inversions required)
       and has a similar accuracy. The matrices
       in this example have been chosen to be $1$-quasiseparable and Hermitian, and we have verified
       the accuracy of the 
       results by means of a direct application of Definition~\ref{def:matrix-func1}. 
       In particular, the timings confirm the almost linear complexity of the procedure.

       \begin{table}
         \centering
         \begin{filecontents*}{expsin.txt}
128 	 1 	  6.79e-03  6 	  2.95e+00 	 6 	   1.33e-13 	 	   1.51e+00 	 7 	   3.30e-14 	 	    
256 	 1 	  2.36e-02  5 	  9.78e+00 	 8 	   4.58e-12 	 	   4.84e+00 	 8 	   1.20e-12 	 	    
512 	 1 	  9.70e-02  5 	  2.46e+01 	 8 	   5.55e-11 	 	   1.22e+01 	 9 	   3.02e-12 	 	    
1024 	 1 	  5.13e-01  13 	  5.70e+01 	 10 	   5.87e-11 	 	   2.35e+01 	 10 	   3.92e-11 	 	    
2048 	 1 	  3.22e+00  5 	  1.32e+02 	 10 	   6.01e-11 	 	   4.81e+01 	 9 	   3.99e-11 	 	    
4096 	 1 	  1.72e+01  4 	  2.45e+02 	 9 	   6.59e-11 	 	   1.27e+02 	 9 	   5.69e-10 		    	 	    	  
\end{filecontents*}
\pgfplotstabletypeset[
   	      columns={0, 4, 6, 7, 9}, 
   	      columns/0/.style={
   		      column name = Size	
   	      },
   	      columns/4/.style={
   	      	column name = $t_{\text{inv}}$,
   	      	postproc cell content/.style={
   	      		/pgfplots/table/@cell content/.add={}{ s}
   	      	}
   	      },
   	      columns/6/.style={
   	      	column name = $\Res_{\text{inv}}$
   	      },
   	      columns/7/.style={
   	      	column name = $t_{\text{sum}}$,
   	      	postproc cell content/.style={
   	      		/pgfplots/table/@cell content/.add={}{ s}
   	      	}
   	      },
   	      columns/9/.style={
   	      	column name = $\Res_{\text{sum}}$
   	      }
         ]{expsin.txt}
         \caption{Timing and accuracy on the computation
         	of the matrix function $f(z) = e^z \sin(z)^{-1}$ 
         	on a $1$-quasiseparable Hermitian matrix $A$ with spectrum contained
         	the unit disc. The residues are measured relatively
         	to the norm of the computed matrix function $f(A)$.}
         \label{tab:sin}
       \end{table}

     \subsection{Validation of the bounds}\label{sec:validation}
     
     This section is devoted to check the accuracy of the
     estimates for the singular values that we have
     proved in the paper. In order to do so we compute
     some matrix function on quasiseparable matrices and
     verify the singular values decay in one large off-diagonal block. In particular, for a matrix of order $m$ --- $m$ even --- we consider the off-diagonal block with row indices from $\frac m2+1$ to $m$ and column indices from $1$ to $\frac m2$. Then, we compare the obtained result
     with the theoretical bound coming from Theorem~\ref{thm:matrix-func-decay}. Notice that Theorem~\ref{thm:matrix-func-decay} provides a family of bounds depending on a parameter $R$ which can be chosen as long as $f(z)$ is holomorphic in $B(0,R)$. So, in every experiment we estimated the $l$-th singular value by choosing the parameter $R$ which provides the tighter bound, among the admissible values for the function $f$ under consideration.

     We choose two particular classes of $1$-quasiseparable matrices for the tests, since we can easily determine
     the bounds on them: 
     \begin{description}
     	\item[Hermitian tridiagonal matrices] These
     	matrices are generated with elements taken from
     	a random Gaussian distribution $N(0,1)$, and
     	are then scaled and
     	  shifted so that their spectrum is contained in a 
     	  ball of center $0$ and radius $\frac 3 4$. These
     	  matrices are normal and the same holds for their
     	  submatrices, so we can avoid the computation
     	  of the constants $\kappa_s(\cdot)$ which are
     	  all equal to $1$. 
     	\item[Hessenberg (scaled) unitary matrices] 
     	  We consider a random unitary matrix which is also
     	  upper Hessenberg, and so in particular it is
     	  $1$-quasiseparable (since unitary matrices
     	  are rank symmetric - the rank of the lower off-diagonal blocks is equal to the corresponding
     	  block above). We then scale the matrices multiplying
     	  by $\frac 3 4$, in order to keep the spectrum on 
     	  the circle of radius $\frac 3 4$. We obtain
     	  these matrices in MATLAB by running
     	  the command
     	  \verb/[A,~] = .75 * qr(hess(randn(N)));/ where $N$ is the chosen 
     	  dimension. 
     \end{description}
     
     As a first example we consider the matrix exponential
     $e^A$ which can be easily computed by means of
     {\tt expm}. We have computed it for many random
     tridiagonal matrices of size $1000 \times 1000$, 
     and the measured and theoretical decays in the submatrix $e^A(501:1000,1:500)$ are 
     reported in  Figure~\ref{fig:exptridiag}. 
     
     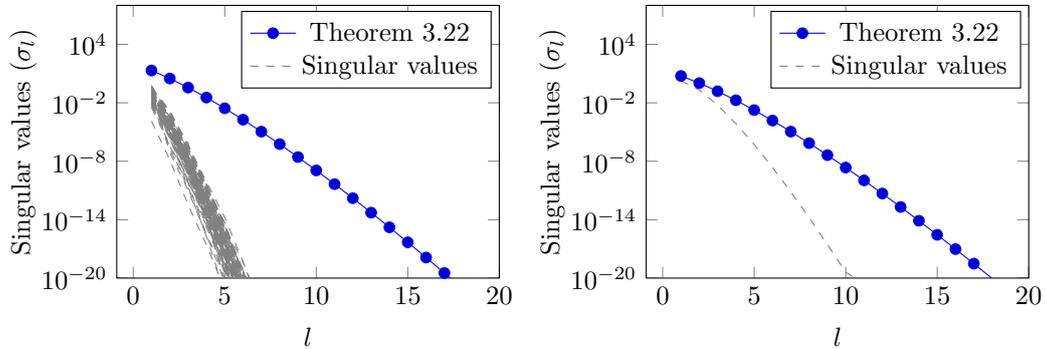
\begin{figure}
     	\centering
     	\begin{tikzpicture}
     	\begin{semilogyaxis}[
     	   xlabel=$l$,ylabel=Singular values ($\sigma_l$),
     	   ymin = 1e-20, xmax = 20,ymax=1e8,
     	   width=.45\linewidth,
     	   height=.25\textheight,
          ytick={1e-20,1e-14,1e-8,1e-2,1e4},   	   
     	   ]
     	\addplot table[x index=0,y index=1] {expbound.txt};
     	\foreach \j in {2, ..., 101} {
   	  	\addplot[gray, no marks,thin,dashed]	table[x index = 0, y index = \j] {expbound.txt};
     	}
     	\legend{Theorem~\ref{thm:matrix-func-decay}, Singular values};
     	\end{semilogyaxis}
     	\end{tikzpicture}
     	\begin{tikzpicture}
     	\begin{semilogyaxis}[
     	xlabel=$l$,ylabel=Singular values ($\sigma_l$),
     	ymin = 1e-20, xmax = 20,ymax=1e8,
     	width=.45\linewidth,
     	height=.25\textheight,
     	ytick={1e-20,1e-14,1e-8,1e-2,1e4}, 
     	]
     	\addplot table[x index=0,y index=1] {exp2bound.txt};
   	\addplot[gray, no marks,thin,dashed]	table[x index = 0, y index = 2] {exp2bound.txt};
     	\legend{Theorem~\ref{thm:matrix-func-decay}, Singular values};
     	\end{semilogyaxis}
     	\end{tikzpicture}
     	\caption{On the left, the bound on the singular values of the off-diagonal matrices of $e^A$ for $100$
     		random Hermitian tridiagonal matrices scaled in order to have spectral radius
     		$\frac 3 4$ are shown. In the right picture the same experiment with a scaled upper
     		Hessenberg unitary matrix is reported (with $1$ matrix only).}
     	\label{fig:exptridiag}
     \end{figure}
     
       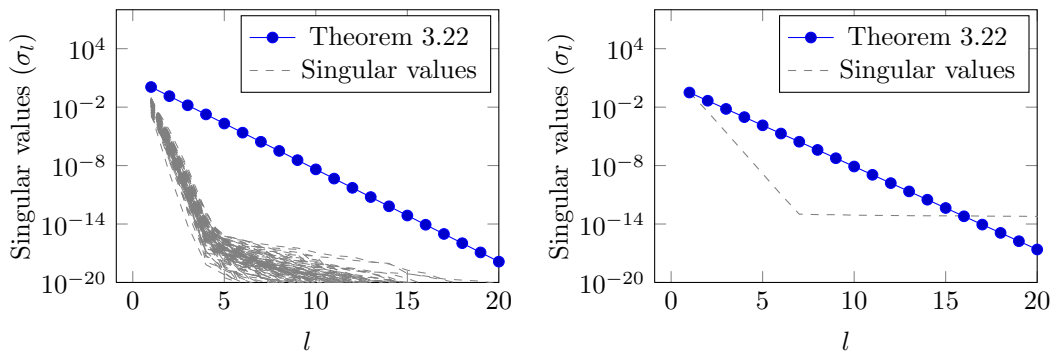
\begin{figure}
       	\centering
       	\begin{tikzpicture}
       	\begin{semilogyaxis}[
       	xlabel=$l$,ylabel=Singular values ($\sigma_l$),
       	ymin = 1e-20, xmax = 20,ymax=1e8,
       	width=.45\linewidth,
       	height=.25\textheight,
           ytick={1e-20,1e-14,1e-8,1e-2,1e4},   	       	
       	]
       	\addplot table[x index=0,y index=1] {logbound.txt};
       	\foreach \j in {2, ..., 101} {
       		\addplot[gray,no marks,very thin,dashed]	table[x index = 0, y index = \j] {logbound.txt};
       	}
       	\legend{Theorem~\ref{thm:matrix-func-decay}, Singular values};
       	\end{semilogyaxis}
       	\end{tikzpicture}~~\begin{tikzpicture}
       	\begin{semilogyaxis}[
       	xlabel=$l$,ylabel=Singular values ($\sigma_l$),
       	ymin = 1e-20, xmax = 20,ymax=1e8,
       	width=.45\linewidth,
       	height=.25\textheight,
       	ytick={1e-20,1e-14,1e-8,1e-2,1e4},   	   
       	]
       	\addplot table[x index=0,y index=1] {log2bound.txt};
       	\addplot[gray, no marks,thin,dashed]	table[x index = 0, y index = 2] {log2bound.txt};
       	\legend{Theorem~\ref{thm:matrix-func-decay}, Singular values};
       	\end{semilogyaxis}
       	\end{tikzpicture}
       	\caption{The picture reports the same experiment of Figure~\ref{fig:exptridiag}, 
       		with the logarithm in place of the exponential. The matrices have however
       		been shifted by $4I$ in order to make the function well-defined. Since this corresponds to evaluating the function
       		$\log(z + 4)$ on the original matrix, one can also find a suitable ball centered
       		in $0$ where the function is analytic.}
       	\label{fig:logtridiag}
       \end{figure}

   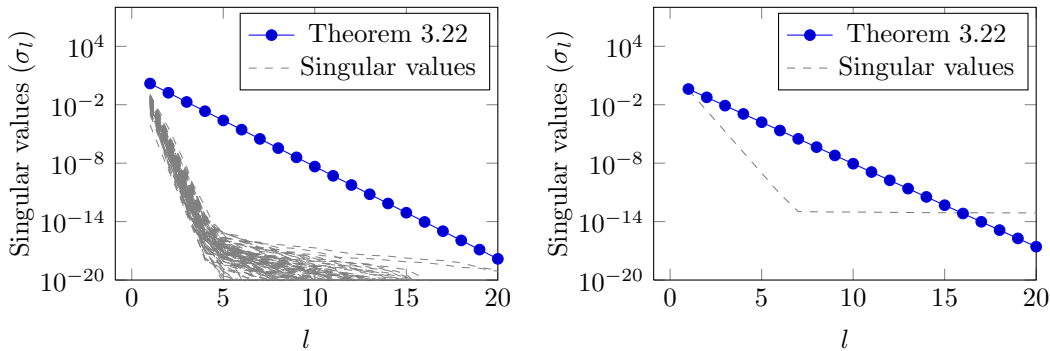
\begin{figure}
   	\centering
   	\begin{tikzpicture}
   	\begin{semilogyaxis}[
   	xlabel=$l$,ylabel=Singular values ($\sigma_l$),
   	ymin = 1e-20, xmax = 20, ymax=1e8,
   	width=.45\linewidth,
   	height=.25\textheight,
   	ytick={1e-20,1e-14,1e-8,1e-2,1e4},   	   
   	]
   	\addplot table[x index=0,y index=1] {sqrtbound.txt};
   	\foreach \j in {2, ..., 101} {
   		\addplot[gray,no marks,very thin,dashed]	table[x index = 0, y index = \j] {sqrtbound.txt};
   	}
   	\legend{Theorem~\ref{thm:matrix-func-decay}, Singular values};
   	\end{semilogyaxis}
   	\end{tikzpicture}~~\begin{tikzpicture}
   	\begin{semilogyaxis}[
   	xlabel=$l$,ylabel=Singular values ($\sigma_l$),
   	ymin = 1e-20, xmax = 20,ymax=1e8,
   	width=.45\linewidth,
   	height=.25\textheight,	
   	ytick={1e-20,1e-14,1e-8,1e-2,1e4},   	   
   	]
   	\addplot table[x index=0,y index=1] {sqrt2bound.txt};
   	\addplot[gray, no marks,thin,dashed]	table[x index = 0, y index = 2] {sqrt2bound.txt};
   	\legend{Theorem~\ref{thm:matrix-func-decay}, Singular values};
   	\end{semilogyaxis}
   	\end{tikzpicture}
   	\caption{In the left picture the bounds on the singular values of the off-diagonal
   		 matrices of $\sqrt{4I + A}$ for $100$
   		random Hermitian tridiagonal matrix scaled in order to have spectral radius
   		$\frac 3 4$ are shown. In the right picture the same experiment is repeated for a
   		scaled and shifted upper Hessenberg unitary matrix.}
   	\label{fig:sqrttridiag}
   \end{figure}    
       
     Similarly, in Figure~\ref{fig:logtridiag} we have reported the analogous experiment concerning the function $\log(4I + A)$. In fact, in order for the logarithm to be well
     defined, we need to make sure that the spectrum of the matrix inside the logarithm
     does not have any negative value. 
     
     As a last example for the tridiagonal matrices we have considered the case of the
     function $\sqrt{4I+A}$, where the matrix has been shifted again in order to obtain
     a reasonable estimate by moving the spectrum away from the branching point. The result for this
     experiment is reported in Figure~\ref{fig:sqrttridiag}. 
     
     In the same figures we have reported also the experiments in the case of the 
     scaled unitary Hessenberg matrix. In this case the variance in the behavior 
     of the singular values was very small in the experiments, and so we have only reported
     one example for each case. 
     
     Notice that while in the symmetric (or Hermitian) case every trailing diagonal submatrix
     is guaranteed to be normal, this is not true anymore for the scaled unitary Hessenberg matrices. 
     Nevertheless, one can verify in practice that these matrices are still not far from normality, 
     and so the bounds that we obtain do not degrade much. 
   \section{Concluding remarks}
   \label{sec:concludingremarks}
   The numerical preservation of the quasiseparable structure when computing a matrix function is an evident phenomenon. Theoretically, this can be explained by the existence of accurate rational approximants of the function over the spectrum of the argument. In this work we have given a closer look to the off-diagonal structure of $f(A)$ providing concrete bounds for its off-diagonal singular values. The off-diagonal blocks have been described as a product between structured matrices with a strong connection with Krylov spaces. This ---combined with polynomial interpolation techniques--- is the key for proving the bounds. 
   
   Moreover, we have developed new tools to deal with the difficulties arising in the treatment of singularities and branching points. In particular the formula of Corollary~\ref{cor:poles} can be employed with the technology of Hierarchical matrices for efficiently computing matrix functions with singularities. An example of this strategy has been provided along with the numerical validation of the bounds.   
     
     \appendix
     \section{Appendix}
   
     \begin{proposition}
       Let $f\in\mathcal C^{\infty}(\mathbb C)$ and $\lambda\in\mathbb C$ then
       \[
       \frac{\partial^{d-1}}{\partial z^{d-1}}\left(\frac{f(z)}{(z-\lambda)^{h+1}} \right)=\frac{(d-1)!}{h!}\sum_{l=1}^{d}(-1)^{l+h+1}\frac{(l+h-1)!}{(d-l)!(l-1)!} f^{(d-l)}(z)(z-\lambda)^{-(h+l)},\qquad \forall d\in\mathbb N^+,h\in\mathbb N.      
       \]
       \end{proposition}
       \begin{proof}
        For every fixed $h\in\mathbb N$ we proceed by induction on $d$. For $d=1$ we get
         \[
         \frac{f(z)}{(z-\lambda)^{h+1}}=\frac{0!}{h!}(-1)^2\frac{h!}{0!0!}\frac{f(z)}{(z-\lambda)^{h+1}}. 
         \]
         For the inductive step, let $d>1$ and observe that
         \begin{align*}
         \frac{\partial^{d}}{\partial z^{d}}\left(\frac{f(z)}{(z-\lambda)^{h+1}} \right)&=\frac{\partial}{\partial z}\left(\frac{\partial^{d-1}}{\partial z^{d-1}}\left(\frac{f(z)}{(z-\lambda)^{h+1}} \right)\right)\\
         &=\frac{\partial}{\partial z}\left(\frac{(d-1)!}{h!}\sum_{l=1}^{d}(-1)^{l+h+1}\frac{(l+h-1)!}{(d-l)!(l-1)!} f^{(d-l)}(z)(z-\lambda)^{-(h+l)} \right)\\
         &=\frac{(d-1)!}{h!}\sum_{l=1}^{d}(-1)^{l+h+1}\frac{(l+h-1)!}{(d-l)!(l-1)!} f^{(d+1-l)}(z)(z-\lambda)^{-(h+l)}\\ &+\frac{(d-1)!}{h!}\sum_{l=1}^{d}(-1)^{l+h+2}(h+l)\frac{(l+h-1)!}{(d-l)!(l-1)!}f^{(d-l)}(z)(z-\lambda)^{-(h+l+1)}\\
         &=\frac{(d-1)!}{h!}\sum_{l=1}^{d}(-1)^{l+h+1}\frac{(l+h-1)!}{(d-l)!(l-1)!} f^{(d+1-l)}(z)(z-\lambda)^{-(h+l)}\\
         &+\frac{(d-1)!}{h!}\sum_{l=2}^{d+1}(-1)^{l+h+1}(h+l-1)\frac{(l+h-2)!}{(d+1-l)!(l-2)!}f^{(d+1-l)}(z)(z-\lambda)^{-(h+l)}\\
         &=\frac{d!}{h!}\sum_{l=1}^{d+1}(-1)^{l+h+1}\frac{(l+h-1)!}{(d+1-l)!(l-1)!} f^{(d+1-l)}(z)(z-\lambda)^{-(h+l)}.
         \end{align*}
         \end{proof} 
   
   \bibliographystyle{abbrv}
   \bibliography{bibliography} 
     
   \end{document}